\documentclass[10pt,a4paper]{article}
\usepackage[utf8]{inputenc}
\usepackage[english]{babel}
\usepackage{amsmath}
\usepackage{amsfonts}
\usepackage{amssymb}
\usepackage[english]{babel}
\usepackage{tikz-cd}
\usepackage{amsthm}
\usepackage{mathrsfs}
\usepackage{xfrac}
\usepackage{mathtools}
\usepackage{extarrows}
\usetikzlibrary{arrows}
\usepackage{array}
\usepackage{url}
\usepackage{geometry}
\usepackage{hyperref}
\usepackage{enumerate}
\usepackage{enumitem}
\numberwithin{equation}{section}
\author{Quentin Faes}

\title{The handlebody group and the images of the second Johnson homomorphism}

{\theoremstyle {definition} \newtheorem {defi} {Définition} [section] }
\newtheorem{definition}[defi]{Definition}
\newtheorem{theorem}[defi]{Theorem}
\newtheorem{question}[defi]{Question}
\newtheorem{coroll}[defi]{Corollary}
\newtheorem{lemma}[defi]{Lemma}
\newtheorem{propo}[defi]{Proposition}
\theoremstyle{definition}
\newtheorem*{ack}{Acknowledgements}
\newtheorem{rem}[defi]{Remark}
\newtheorem{example}[defi]{Example}
\usetikzlibrary{arrows}
\usepackage{indentfirst}
\makeatletter
\def\blfootnote{\xdef\@thefnmark{}\@footnotetext}
\usepackage{tocloft,calc}
\makeatother
\usepackage[titletoc]{appendix}
\makeatother

\usepackage{titlesec}
\newcommand{\address}{{
  \bigskip
  \footnotesize

  \textsc{Institut de mathématiques de Bourgogne, UMR 5584, Université Bourgogne Franche-Comté, 21000 Dijon, France}\par\nopagebreak
  \textit{E-mail address}  \texttt{quentin.faes@u-bourgogne.fr}}}
\usepackage[runin]{abstract}
\abslabeldelim{.}
\definecolor{wqwqwq}{rgb}{0,0,0}
\usepackage{pgf,tikz,pgfplots}
\pgfplotsset{compat=1.15}
\newcommand{\AJ}{\mathcal{A} \cap J_2}
\newcommand{\AJk}{\mathcal{A} \cap J_k}
\newcommand{\BJk}{\mathcal{B} \cap J_k}
\newcommand*\circled[1]{\tikz[baseline=(char.base)]{
            \node[shape=circle,draw,inner sep=2pt] (char) {#1};}}

\newcommand{\ltree}[4]{\tikz[baseline=12pt, scale = 0.7]{
\draw [color=wqwqwq] (0,2-1.25)-- (2,2-1.25);
\draw [color=wqwqwq] (0,1.5)-- (0,0);
\draw [color=wqwqwq] (2,1.5)-- (2,0);

\draw[color=wqwqwq] (0,3-1.25) node {#1};
\draw[color=wqwqwq] (0,1.1-1.25) node {#2};
\draw[color=wqwqwq] (2,1.1-1.25) node {#3};
\draw[color=wqwqwq] (2,3-1.25) node {#4};}}

\newcommand{\vltree}[4]{\tikz[baseline=6pt, scale = 0.4]{
\draw [color=wqwqwq] (0,2-1.25)-- (2,2-1.25);
\draw [color=wqwqwq] (0,1.5)-- (0,0);
\draw [color=wqwqwq] (2,1.5)-- (2,0);

\draw[color=wqwqwq] (0,3-1.25) node {#1};
\draw[color=wqwqwq] (0,1.1-1.25) node {#2};
\draw[color=wqwqwq] (2,1.1-1.25) node {#3};
\draw[color=wqwqwq] (2,3-1.25) node {#4};}}

\newcommand{\ltritree}[3]{\tikz[baseline=12pt ,scale = 0.7]{
\draw [color=wqwqwq] (1,1.5)-- (1,0.75);
\draw [color=wqwqwq] (1,0.75)-- (1-0.866*0.75,0.75-.37);
\draw [color=wqwqwq] (1,0.75)-- (1+0.866*0.75,0.75-.37);

\draw[color=wqwqwq] (1,1.5+0.2) node {#1};
\draw[color=wqwqwq] (1-0.866*0.75-0.2,0.75-.37-0.2) node {#2};
\draw[color=wqwqwq] (1+0.866*0.75+0.2,0.75-.37-0.2) node {#3};}}
\usepackage{changepage}
\newcommand{\ses}[3]{\[ 0 \longrightarrow #1 \longrightarrow #2 \longrightarrow #3 \longrightarrow 0\]}

\newcommand{\K}{\mathcal{K}}
\newcommand{\I}{\mathcal{I}}
\newcommand{\A}{\mathcal{A}}
\newcommand{\G}{\mathcal{G}}
\newcommand{\B}{\mathcal{B}}
\newcommand{\M}{\mathcal{M}}
\newcommand{\Q}{\mathbb{Q}}

\begin{document}
\titlelabel{\thetitle.  }
\maketitle
\begin{adjustwidth}{20 pt}{20pt}
\textsc{Abstract.}  Given an oriented surface bounding a handlebody, we study the subgroup of its mapping class group defined as the intersection of the handlebody group and the second term of the Johnson filtration: $\AJ$. We introduce two trace-like operators, inspired by Morita's trace, and show that their kernels coincide with the images by the second Johnson homomorphism $\tau_2$ of $J_2$ and $\AJ$, respectively. In particular, we answer by the negative to a question asked by Levine about an algebraic description of $\tau_2( \AJ)$. By the same techniques, and for a Heegaard surface in $S^3$, we also compute the image by $\tau_2$ of the intersection of the Goeritz group $\mathcal{G}$ with $J_2$.
\end{adjustwidth}
\setcounter{tocdepth}{1}
\renewcommand{\contentsname}{\hfill Contents\hfill}
\renewcommand{\cfttoctitlefont}{\scshape}
\renewcommand{\cftaftertoctitle}{\hfill}
\renewcommand{\cftsecaftersnum}{.}
\setlength{\cftbeforesecskip}{0.05cm}
\tableofcontents
\section{Introduction and notations}
\label{sec1}
We consider an abstract handlebody $V_{g}$ of genus $g$ whose boundary is a surface $\Sigma_{g}$ of genus $g$. This surface minus a disk will be the surface with non-empty boundary $\Sigma_{g,1}$. We will often forget the indices concerning the genus and the number of boundary components when they are clear from context. \blfootnote{This research has been supported by the project “AlMaRe” (ANR-19-CE40-0001-01) of the ANR and the project ``ITIQ-3D" of the Région Bourgogne Franche-Comté.}

The study of the handlebody group $\mathcal{A}$ is of major importance for the study of the mapping class group of surfaces $\mathcal{M}$, especially in connection with the theory of 3-manifolds and their Heegaard presentations. The reader may find useful information on this topic in the survey by Hensel \cite{hensel}. It is a non-normal subgroup of the mapping class group of infinite index, which makes its study as a subgroup of $\mathcal{M}$ uneasy. Precisely, $\mathcal{M}$ will be our notation for $\mathcal{M}_{g,1}$, the mapping class group of $ \Sigma_{g,1}$, and $\mathcal{A}$ will be our notation for $\mathcal{A}_{g,1}$, the mapping class group of $V_g$ relative to a disk in $\partial V_g$.

We will denote $\pi := \pi_1(\Sigma_{g,1} , x_0)$, where $x_0$ is a point on the boundary of $\Sigma_{g,1}$, and $H :=H_1(\Sigma_{g,1})$ its abelianization. Recall that  $\pi$ is isomorphic to the free group with $2g$ generators $F_{2g}$, and hence $H$ is isomorphic to $\mathbb{Z}^{2g}$. The curves $(\alpha_i)_{1 \leq i \leq g}$ and $(\beta_i)_{1 \leq i \leq g}$ on Figure \ref{surface} are two cutting systems such that each curve in the first one has exactly one intersection point with exactly one curve in the second one, and vice versa. Such a choice is called a system of \emph{meridians} and \emph{parallels}. In particular, it fixes a choice of a basis for $H = \mathbb{Z}\langle a_1,a_2 \dots a_g, b_1, b_2, \dots b_g\rangle$, where $a_i$ (resp. $b_i$) is the homology class of $\alpha_i$ (resp. $\beta_i$). When $\Sigma_{g,1}$ will be regarded as the boundary of $V_g$ (minus a disk), we will suppose that the meridians (i.e. the curves $\alpha_i$) bound pairwise-disjoint disks in the handlebody. If promoted to elements of the fundamental group $\pi$, the curves $\beta_i$ define generators of $\pi' :=  \pi_1(V , x_0)$ and the curves $\alpha_i$ normally generate the kernel of the surjection $\pi \rightarrow \pi'$ induced by the inclusion of $\Sigma_{g,1}$ in $V_g$. We denote $\mathbb{A}$ this kernel, so that $\pi' \simeq \pi / \mathbb{A}$. It is well-known  that the handlebody group $\mathcal{A}$, which can be thought of as consisting of elements of the mapping class group $\mathcal{M}$ extending to the whole handlebody, coincides with the subgroup of $\mathcal{M}$ preserving $\mathbb{A}$ \cite{hensel}. We emphasize that, from the point of view of the surface $\Sigma_{g,1}$, this subgroup $\mathcal{A}$ of $\mathcal{M}$ depends on the choice of handlebody $V_g$.

\begin{figure}[h]
	\centering
	\includegraphics[scale= 0.23]{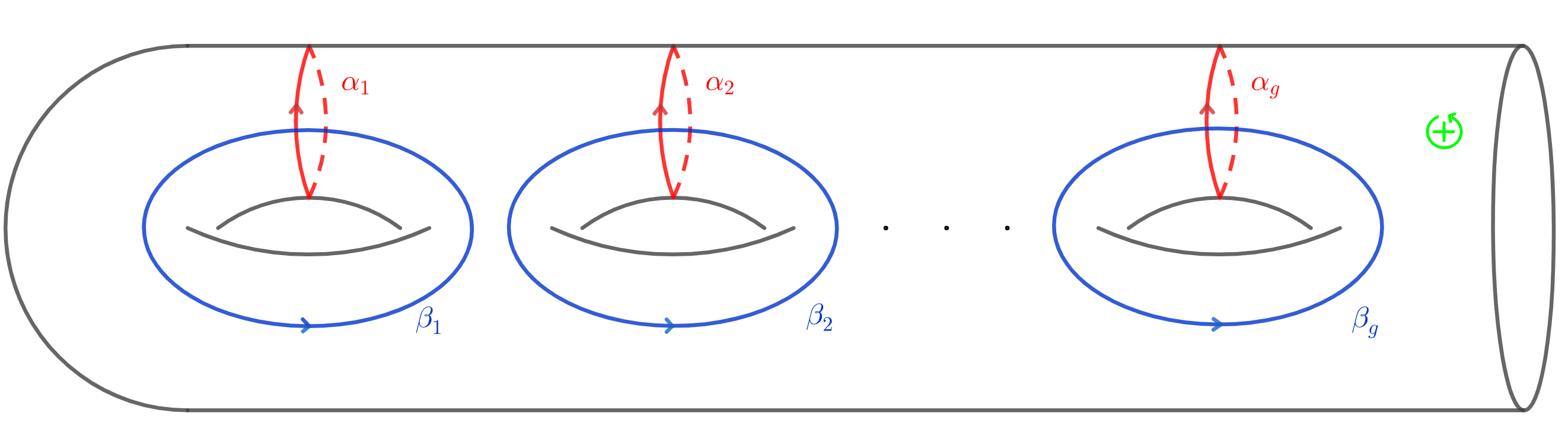}
	\caption{Model for $\Sigma_{g,1}$, and a possible choice of system of meridians and parallels}
	\label{surface}
\end{figure}
We also consider $H':= H_1(V_g)$ the first homology group of the handlebody. The kernel of the homomorphism $H \rightarrow H'$ induced by the inclusion of $\Sigma_{g,1}$ in $V_g$ is denoted $A$. It is generated in $H$ by the elements $a_i$. The group $H' \simeq H/A$ is freely generated by the classes of the elements $b_i$, but should not be thought of as a subgroup of $H$ since there is no canonical way to choose a supplement of $A$ in $H$. We consider the homological intersection form $\omega : H \otimes H \rightarrow \mathbb{Z}$, which induces a non-singular pairing $\omega' : A \otimes H' \rightarrow \mathbb{Z}$. We denote by $\mathcal{L}(H)= \bigoplus_{k \geq 1} \mathcal{L}_k(H)$ the graded Lie ring freely generated by $H$ in degree 1. We denote by $T(H)$ the tensor algebra, in which  $\mathcal{L}(H)$ can be imbedded. The symmetric algebra $S(H)$ is as usual the quotient of $T(H)$ by its antisymmetric tensors.

 In this paper we focus on the study of the group $\AJ$, where $J_2$ is the second term of the Johnson filtration $(J_k)_{k \geq 1}$ \cite{johsurvey}. Examining the group $\AJ$ seems natural when one uses Johnson-type homomorphisms to study finite-type invariants of 3-manifolds from the point of view of Heegaard splittings. Besides, the Johnson filtration of $\mathcal{M}$ is separating, and so is its intersection with $\mathcal{A}$: hence the study of the filtration $(\AJk)_{k\geq 1}$, including the determination of its associated graded ${\bigoplus}_{k\geq 1} \frac{\AJk}{\mathcal{A} \cap J_{k+1}}$, is also relevant for the study of the group $\mathcal{A}$ itself. As the Torelli group $\mathcal{I}$ (the subgroup of $\mathcal{M}$ acting trivially at the homological level) is the first term $J_1$ of the Johnson filtration, the question addressed here is the next natural step after the study of $\mathcal{A} \cap \mathcal{I}$ pursued by Omori in \cite{omo}, and the earlier computation of $\frac{\mathcal{A}\cap J_1}{\AJ}$ given by Morita in \cite{mor}. 
 
The study of the relationship between the Johnson filtration and the handlebody group may cover other aspects. In particular, it was proved independently by Hain \cite{hain} and Jorgensen \cite{jor} that there exist elements of $\M$ arbitrarily deep in the Johnson filtration that are not in the union of the conjugates of $\A$ in $\M$. Besides, Hain also introduced a filtration of a completion of $\M$ (relative to the symplectic representation), called the \emph{weight filtration} and he introduced in \cite{hain} another filtration, the \emph{relative weight filtration} associated to the choice of a handlebody bounded by $\Sigma$. The study of the graded spaces associated to these filtrations should be related to the quotients $\frac{\AJk}{\mathcal{A} \cap J_{k+1}} \otimes \mathbb{Q}$. 
 
In this paper, we work with coefficients in $\mathbb{Z}$ (the only exception will be in Appendix \ref{appA}). To get a more precise grasp of the intersection $\AJ$, we use the Johnson homomorphisms $(\tau_k)_{k \geq 1}$ introduced in \cite{johsurvey}, trace-like operators, and the Casson invariant. 

The first step is to define a trace-like operator $\operatorname{Tr}^{as}$ on the codomain of $\tau_2$ (which is the group of symplectic derivations of degree 2 of $\mathcal{L}(H)$, denoted $D_2(H)$). Using the results of Morita \cite{mor} and Yokomizo \cite{yok}, we prove that the kernel of $\operatorname{Tr}^{as}$ is precisely $\tau_2(J_2)$. We also show that $\tau_2([J_1,J_1]) = \operatorname{Ker}(\operatorname{Tr}^{sym})$, where $\operatorname{Tr}^{sym}$ is another trace-like map (defined on a subgroup of $D_2(H)$). The codomains of $\operatorname{Tr}^{as}$ and  $\operatorname{Tr}^{sym}$ will be respectively $\operatorname{Ker}(\omega : \Lambda^2(H/2H) \rightarrow \mathbb{Z}_2 )$ and $\operatorname{Ker}(\omega : S^2(H/2H) \rightarrow \mathbb{Z}_2 )$. Here, and in the sequel, for any module~$V$, the notation $S^2(V)$ stands for the quotient of $V \otimes V$ by the two-sided ideal generated by the tensors of the form $v \otimes w - w \otimes v$. The module $\Lambda^2(V)$ is the quotient of the same module by the two-sided ideal generated by the tensors of the form $v \otimes w + w \otimes v$ and $v \otimes v$. Notice that this last type of tensor is needed in the definition in the case of $\mathbb{Z}_2$-modules. For example there is a canonical projection from $S^2(H/2H)$ to $\Lambda^2(H/2H)$, given by reducing the classes of elements of the form $v \otimes v$.

The second step is the study of $\tau_2(\AJ)$, which, by definition of $\tau_2$, is isomorphic to $\frac{\AJ}{\mathcal{A} \cap J_3}$. In \cite{lev}, Levine observed that this image is contained  in the kernel of the canonical projection from $D_2(H)$ to $D_2(H')$. He asked whether the intersection of $\operatorname{Ker}(D_2(H) \rightarrow D_2(H'))$ with $\operatorname{Im}(\tau_2)$ was equal to $\tau_2(\AJ)$. We shall define, using the non-singular pairing $\omega'$, another trace-like operator $\operatorname{Tr}^{A}$ vanishing on $\tau_2(\AJ)$, but not on this subgroup proposed by Levine. Therefore, we answer negatively to Levine's question. Furthermore, $\operatorname{Tr}^{as}$ and  $\operatorname{Tr}^{A}$ will allow us to compute precisely $\tau_2(\AJ)$, and thus to identify $\frac{\AJ}{\mathcal{A} \cap J_3}$ with an explicit subgroup of $D_2(H)$.

The paper is organized as follows. In Section \ref{sec2}, we review the definition of the Johnson filtration $(J_k)_{k \geq 1}$ from \cite{johsurvey}, as well as the definition of the Johnson homomorphisms $(\tau_k)_{k \geq 1}$ from \cite{mor93}. Then we define the maps $\operatorname{Tr}^{as}$ and $\operatorname{Tr}^{sym}$ and use them to characterize $\tau_2(J_2)$ and $\tau_2([J_1,J_1])$, respectively. In Section \ref{sec3}, we first review closely related works. Then we recall the definition of the Levine filtration $(L_k)_{k \geq 1}$ from \cite{lev}, so as to state and motivate precisely the question asked by Levine. In Section~\ref{sec4}, we define the map $\operatorname{Tr}^{A}$, and we prove that it gives a new obstruction for an element of $D_2(H)$ to be in $\tau_2(\mathcal{A}\cap J_2)$, by using Morita's decomposition of the Casson invariant \cite{mor}. In Section \ref{sec5}, we compute the image $\tau_2(\mathcal{A} \cap J_2)$ using the algebraic tools introduced in Sections \ref{sec2} and \ref{sec4}. In Section \ref{sec6}, when $\Sigma$ is a Heegaard surface of  $S^3$, we compute $\tau_2(\mathcal{G}\cap J_2)$ where $\mathcal{G} \subset \mathcal{M}$ is the Goeritz group defined by this Heegaard splitting. Finally, in Appendix A, we decompose $\tau_2(\mathcal{G}\cap J_2) \otimes \Q$ into irreducible $\mathrm{GL}(g, \mathbb{Q})$-modules, and we check the computation of Section \ref{sec6} for rational coefficients, without using the main result of Section \ref{sec5}.
\begin{ack}
I would like to thank my advisor, Gwénaël Massuyeau, for his careful rereadings and his encouragements. I am deeply greatful to Anderson Vera, for his helpful comments and for giving the idea of the computation in Section \ref{sec6}. I also thank Richard Hain for giving his comments on the first version of this paper. \end{ack}
\section{Image of the second Johnson homomorphism $\tau_2 $}
\label{sec2}

\subsection{The space of symplectic derivations of degree 2}
Here, we review some facts about Johnson homomorphisms and their diagrammatic description. We are especially interested in describing the image of the second Johnson homomorphism.

\subsubsection{Johnson homomorphisms and tree-like Jacobi diagrams}
The Johnson filtration and the Johnson homomorphisms have been introduced and studied by Johnson and Morita in \cite{johsurvey,mor93}. Recall that $\pi := \pi_1(\Sigma_{g,1})$ is a free group. For $k \geq 1$, we consider its lower central series $(\Gamma_k \pi)_{k \geq 1}$. We call the quotient $N_k := \pi / \Gamma_{k+1} \pi$ the $k$-th nilpotent quotient of $\pi$. The first nilpotent quotient is canonically isomorphic to $H := H_1(\Sigma_{g,1})$. It is clear that $\mathcal{M}$ acts both on $\pi$ and all its nilpotent quotients. There is an exact sequence: \ses{\mathcal{L}_{k+1}(H)}{N_{k+1}}{N_{k}} where the first non-trivial arrow is given by the identification between $\mathcal{L}_{k+1}(H)$ and $\Gamma_{k+1} \pi / \Gamma_{k+2} \pi$. This sequence induces the short exact sequence :
\[ 0 \longrightarrow \operatorname{Hom}(H,\mathcal{L}_{k+1}(H)) \longrightarrow \operatorname{Aut}(N_{k+1}) \longrightarrow \operatorname{Aut}(N_{k}). \]
 \noindent The group $J_k$ is defined as the kernel of the canonical homomorphism $\rho_k : \mathcal{M} \rightarrow \operatorname{Aut}(N_k)$. In particular $J_1$ is called the Torelli group, otherwise denoted $\mathcal{I} = \mathcal{I}_{g,1}$. It consists of elements of the mapping class group acting trivially on the homology of the surface. The alternative notation $\mathcal{K} = \mathcal{K}_{g,1}$ is also sometimes used for $J_2$.

The restriction of $\rho_{k+1}$ to $J_k$ then induces a morphism: \[ \tau_k : J_k \longrightarrow \operatorname{Hom}(H,\mathcal{L}_{k+1}(H)). \]
We call this map the \textit{$k$-th Johnson homomorphism}. Its kernel is $J_{k+1}$. Furthermore, the mapping class group acts on itself by conjugation, inducing an action of the symplectic group $\operatorname{Sp}(H)$ on the quotient $J_k/ J_{k+1}$. This group also naturally acts on $H$. Each $\tau_k$ is then $\operatorname{Sp}(H)$-equivariant. It is also known that the graded space induced by the Johnson filtration has a Lie structure, its bracket being induced by the commutator in $\mathcal{M}$. The target space of $\tau_k$ can be identified with the space of \textit{ derivations of degree $k$}, i.e. derivations of $\mathcal{L}(H)$ mapping $H = \mathcal{L}_1(H)$ to $\mathcal{L}_{k+1}(H)$. We denote by $D_k(H)$ the subspace of \textit{symplectic derivations} of degree $k$. It consists of derivations of degree $k$ sending $\tilde{\omega} \in \Lambda^2 H \simeq  \mathcal{L}_2(H)$, the bivector dual to $\omega$, to 0. The fact that an element of $\mathcal{M}$ fixes the boundary of $\Sigma_{g,1}$ allows to further restrict the image of $\tau_k$ to $D_k(H)$. Also, $D_k(H)$ is determined by the short exact sequence: \ses{D_k(H)}{H \otimes \mathcal{L}_{k+1}(H)}{\mathcal{L}_{k+2}(H)} \noindent where the arrow from ${H \otimes \mathcal{L}_{k+1}(H)}$ to ${\mathcal{L}_{k+2}(H)}$ is the bracket of the free Lie algebra.

With these definitions, the spaces $(D_k(H))_{k \geq 1}$ reassembles in a graded Lie algebra $D(H)$ (the bracket of two derivations $d_1$ and $d_2$ being classically defined as $ d_1 d_2 -d_2 d_1$). The family $(\tau_k)_{k \geq 1}$ induces a map $\tau$ :
\[ \tau : \underset{k\geq 1}{\bigoplus} J_k / J_{k+1} \longrightarrow D(H)\]

\noindent which is an $\operatorname{Sp}(H)$-equivariant graded Lie morphism. The map $\tau_k$ is not onto $D_k(H)$ in general, but it is known to be surjective for $k=1$ \cite{joh80} and rationally surjective for $k=2$~ \cite{mor}. We shall describe in the next subsections the image of $\tau_2$ in a precise way.

We also need to define the spaces of tree-like Jacobi diagrams $\mathcal{A}_k^t(H)$ and rooted tree-like Jacobi diagrams $\mathcal{A}_k^{t,r}(H)$. A tree is a connected graph that is contractible as a topological space. From now on, by ``a tree", we mean a uni-trivalent tree $T$, possibly rooted, whose set of trivalent (or \emph{internal}) vertices is oriented (the orientation being counterclockwise in all the figures), and whose set of univalent (or \emph{external}) vertices, denoted $v_1(T)$, is colored by elements of $H$. We will also refer to external vertices as \emph{leaves} and internal vertices as \emph{nodes}. The cardinal of the set of trivalent vertices $v_3(T)$ is the \textit{degree} of the tree $T$. The spaces $\mathcal{A}_k^t(H)$ and $\mathcal{A}_k^{t,r}(H) $ are the $\mathbb{Z}$-modules generated by trees (respectively rooted trees) of degree $k$ subject to some relations: multilinearity of the labels, the \textit{AS relation}, and the \textit{IHX relation}. We specify these relations for $k=2$ in Figure \ref{figIHX}, and we refer the reader to \cite{levtrees} for further details about what follows. These spaces assemble in two graded algebras $\mathcal{A}^t(H)$ and $\mathcal{A}^{t,r}(H) $ endowed respectively with a Lie bracket and a quasi-Lie bracket. For the bracket of $\mathcal{A}^t(H)$, take two trees, and sum all the ways to contract an external vertex from the first one with an external vertex from the second one using the symplectic form $\omega$. For $\mathcal{A}^{t,r}(H)$, take two trees, and form a tree by gluing their roots  to a rooted binary tree with two leaves. 

\begin{figure}[h]
\centering
$IHX$ : $\ltree{$a$}{$b$}{$c$}{$d$} = \ltree{$a$}{$d$}{$c$}{$b$} + \ltree{$a$}{$c$}{$b$}{$d$}$

$AS$ : $\ltree{$a$}{$b$}{$c$}{$d$} = -\ltree{$b$}{$a$}{$c$}{$d$}$

\caption{Relations in $\mathcal{A}^t_2 (H)$}
\label{figIHX}
\end{figure}

We also define, for any $k$, maps
\begin{align*}
 \eta_k : \mathcal{A}^t_k (H) & \longrightarrow D_k(H) \\  T &  \longmapsto  \sum_{x\in v_1(T)} l_x \otimes T^x
\end{align*}
\noindent where $l_x$ is the element of $H$ coloring the vertex $x$ and $T^x$ is the rooted tree obtained by setting $x$ to be the root in $T$, read as an element of $\mathcal{L}_{k+1}(H)$ (which can be done inductively by considering that $ \ltritree{*}{$a$}{$b$}$ corresponds to $[b,a]$). These maps assemble into a graded Lie algebra morphism which we refer to as ``the expansion map".

\subsubsection{A presentation for $D_2(H)$}

The first Johnson homomorphism takes values in $D_1(H)$ which is known to be isomorphic to $\Lambda^3H$. The map $\tau_1$ is surjective, and $\eta_1$ is an isomorphism, thus identifying the quotient $J_1/J_2$ to $\mathcal{A}^t_1 (H)$.

The second Johnson homomorphism takes values in $D_2(H)$. This space is well understood too. Morita \cite{mor}, using the exact sequence  \[ 0 \longrightarrow \Lambda^3 H \longrightarrow H \otimes \mathcal{L}_{2}(H) \longrightarrow  \mathcal{L}_{3}(H) \longrightarrow   0 ,\]
described it as the image of $(\Lambda^2 (H) \otimes \Lambda^2 H)^{\mathfrak{S}_2}$ in the quotient $(H \otimes H \otimes \Lambda^2 H)/ H \otimes \Lambda^3 H$, where $\mathcal{L}_{2}(H)$ has been identified with $\Lambda^2 H$. 

We will prefer to use the following description given by Levine \cite{levtrees}. Indeed, a simpler way to think about this space is to use the free \emph{quasi-Lie algebra} $\mathcal{L}'(H) = \bigoplus_{k \geq 1}\mathcal{L}_k'(H) $ on $H$, which is defined similarly to the free Lie algebra with the alternativity axiom $[x,x] = 0$ (for any $x \in \mathcal{L}$) replaced by the antisymmetry axiom $[x,y] + [y,x]= 0$ (for any $x,y \in \mathcal{L}$). This change adds 2-torsion to the group. We define $D'_k(H)$, similarly to $D_k(H)$, as the kernel of the bracket from $H \otimes \mathcal{L}'_{k+1}(H)$ to $\mathcal{L}'_{k+2}(H)$. We will only use $k=1$ or $2$ in this paper. We have $D'_1(H) \simeq D_1(H)$ and a commutative diagram with exact rows: 
\[
\begin{tikzcd} 0  \arrow[r] & D'_2(H) \arrow[r]\arrow[d] & H \otimes \mathcal{L}'_{3}(H) \arrow[r]\arrow[d] & \mathcal{L}'_{4}(H) \arrow[r]\arrow[d]  & 0 \\ 0  \arrow[r] & D_2(H) \arrow[r] & H \otimes \mathcal{L}_{3}(H) \arrow[r] & \mathcal{L}_{4}(H)\arrow[r]  & 0. \end{tikzcd}
\]

\noindent Levine also showed that we have the following exact sequence:

\begin{equation}
0\longrightarrow D'_2(H) \longrightarrow D_2(H) \longrightarrow  \Lambda^2(H/2H) \longrightarrow   0.
\label{eqlev}
\end{equation}

This is helpful for the following reason: $D_2(H)$, which is a free abelian group, can be thought of as a lattice in $D_2(H) \otimes \mathbb{Q}$.
By \eqref{eqlev}, to generate $D_2(H)$, one simply needs to add to $D'_2(H)$ expansions of type $\frac{1}{2}\eta(u-u)$ for any rooted tree $u$ with 2 external vertices, that we glue to its copy along their roots. These are indeed elements of $D_2(H)$, i.e. they have integer coefficients. For $x,y \in \Lambda^2 H$ we write $x \leftrightarrow y$ for the element $x \otimes y + y \otimes x$. Also $\Lambda^4 H$ can be embedded in $\left(\Lambda^{2} H \leftrightarrow \Lambda^{2} H\right) \subset\left(\Lambda^{2} H \otimes \Lambda^{2} H\right)^{\mathfrak{S}_{2}}$ by sending $a \wedge b \wedge c \wedge d$ to \[(a \wedge b) \leftrightarrow(c \wedge d)-(a \wedge c) \leftrightarrow(b \wedge d)+(a \wedge d) \leftrightarrow(b \wedge c), \] for $a,b,c,d \in H$.
It has been proven by Levine in \cite{levtrees} using Morita's work in \cite{mor} (see also \cite[Prop. 3.1]{masY3}) that the map
\begin{align*}
\frac{(\Lambda^2 H \otimes \Lambda^2 H)^{\mathfrak{S}_2}}{\Lambda^4 H} &\longrightarrow D_2(H) \\
(a \wedge b) \leftrightarrow (c \wedge d) &\longmapsto a \otimes[b,[c, d]]+b \otimes[[c, d], a]\\ & \quad +c \otimes[d,[a, b]]+d \otimes[[a, b], c] \\ & \quad= \eta_2 \Big (\ltree{$a$}{$b$}{$c$}{$d$}\Big) \\ (a \wedge b) \otimes (a\wedge b) & \longmapsto a \otimes[b,[a, b]]+b \otimes[[a, b], a] \\ &\quad= \frac{1}{2}\eta_2 \Big(\ltree{$a$}{$b$}{$a$}{$b$} \Big)
\end{align*} 
\noindent is a well-defined isomorphism that fits in the commutative diagram with exact rows
\begin{equation}
\begin{tikzcd} 0  \arrow[r] & \frac{S^{2} (\Lambda^{2} H)}{\Lambda^{4} H} \arrow[r, "\leftrightarrow"]\arrow[d , "\eta '"] & \frac{\left(\Lambda^{2} H \otimes \Lambda^{2} H\right)^{\mathfrak{S}_{2}}}{\Lambda^{4} H} \arrow[r]\arrow[d] & \frac{\Lambda^{2} H}{2 \cdot \Lambda^{2} H} \arrow[r]\arrow[d]  & 0 \\ 0  \arrow[r] & D'_2(H) \arrow[r] & D_2(H) \arrow[r] & \Lambda^2(H/2H)\arrow[r]  & 0 \label{diagmas}\end{tikzcd}
\end{equation}
\noindent where $\eta '$  is defined in a way similar to $\eta$ \cite{levtrees}. To be precise the expansion of a tree is actually an element of $D'_2(H)$, and this defines an isomorphism between $ D'_2(H)$ and $\mathcal{A}^t_2(H)$ \cite{levtrees}.

From this we deduce the following presentation of the abelian group $D_2(H)$.

\begin{propo}
$D_2(H)$ is generated by trees $\ltree{a}{b}{c}{d}$ for $a,b,c$ and $d$ in $H$ and elements $a \odot b$ for $a,b \in H$ subject to the following relations:
\newline

- $AS$, $IHX$, and multilinearity with respect to the labels for all trees

- $a \odot a= 0$ and $a \odot b= b \odot a$ for all  $(a,b) \in H \times H$

- $2(a \odot b)= \ltree{$a$}{$b$}{$a$}{$b$}$

- $(a+b) \odot c = a \odot c + b \odot c +  \ltree{a}{c}{b}{c}$

\label{prop2.1}
\end{propo}

\begin{proof}
Let us momentarily denote by $G$ the group defined by the presentation. We define a homomorphism from $G$ to $\frac{\left(\Lambda^{2} H \otimes \Lambda^{2} H\right)^{\mathfrak{S}_{2}}}{\Lambda^{4} H}$ by sending $a \odot b$ to the class of $(a \wedge b) \otimes (a \wedge b)$ and any tree $\ltree{$a$}{$b$}{$c$}{$d$}$ to the element corresponding to its expansion through diagram \eqref{diagmas}, i.e. to $(a \wedge b) \leftrightarrow (c \wedge d)$. We define a converse homomorphism by reversing the previous mappings. It  suffices to show that these maps are well-defined to conclude. It is straightforward calculus to check that the relations for $G$ vanish in $\frac{\left(\Lambda^{2} H \otimes \Lambda^{2} H\right)^{\mathfrak{S}_{2}}}{\Lambda^{4} H}$, noting in particular that it is known that the expansion map sends the $IHX$ relation to $0$. Conversely, ${\left(\Lambda^{2} H \otimes \Lambda^{2} H\right)^{\mathfrak{S}_{2}}}$ can be presented in the following way. The group ${\left(\Lambda^{2} H \otimes \Lambda^{2} H\right)^{\mathfrak{S}_{2}}}$ is generated by elements $(a \wedge b) \otimes (a \wedge b)$ and elements $(a \wedge b \leftrightarrow c \wedge d) $ with $a,b,c,d \in H$. The relations are $(a \wedge b) \leftrightarrow (a \wedge b) = 2 (a \wedge b) \otimes (a \wedge b)$  and $((a+b) \wedge c) \otimes ((a+b) \wedge c) - (a \wedge c) \otimes (a \wedge c) -(b \wedge c) \otimes (b \wedge c) = (a \wedge c \leftrightarrow b \wedge c)$. This presentation is summarized in the short exact sequence \ses{S^2(\Lambda^2 H)}{{\left(\Lambda^{2} H \otimes \Lambda^{2} H\right)^{\mathfrak{S}_{2}}}}{\frac{\Lambda^{2} H}{2 \cdot \Lambda^{2} H}} where the last arrow sends $(a \wedge b) \otimes (a \wedge b)$ to $a\wedge b$ and $(a \wedge b) \leftrightarrow (c \wedge d)$ to $0$. We can then read these relations in the presentation of $G$. We finally notice that for any $a,b,c,d \in H$, $(a \wedge b) \leftrightarrow(c \wedge d)-(a \wedge c) \leftrightarrow(b \wedge d)+(a \wedge d) \leftrightarrow(b \wedge c)$ is sent to the $IHX$ relation, up to some antisymmetries.
\end{proof}

\begin{rem} Elements $a \odot b$ for $a,b\in H$ correspond to halfs of symmetric trees (namely $\frac{1}{2}\ltree{$a$}{$b$}{$a$}{$b$}$ for $a,b \in H$) through the inclusion $D_2(H) \subset D_2(H) \otimes \mathbb{Q} \simeq \mathcal{A}^t_2(H) \otimes \mathbb{Q}$. Then, a concise and simple way to summarize the previous discussion, is to say that $D_2(H)$ embeds in the space of trees $\mathcal{A}^t_2(H) \otimes \mathbb{Q}$, and its image is the lattice generated by trees and halfs of symmetric trees. This is what we will do, especially in Sections \ref{sec4} and \ref{sec5}.
\label{remtree}
\end{rem}

\subsection{An explicit description of $\operatorname{Im}(\tau_2)$ in $D_2(H)$}
We aim at a homomorphism that would be explicitly defined on $D_2(H)$, using the presentation in Proposition \ref{prop2.1}, and whose kernel would be $\operatorname{Im}(\tau_2)$. From now on, we will abuse notation and identify $D_2'(H)$ with $\mathcal{A}^t_2(H)$ and think of its elements as trees.

In \cite{joh2}, Johnson showed that $\mathcal{K}$ is generated by Dehn twists along bounding simple closed curves (called \emph{BSCC} maps) of genus 1 and 2. We will denote $T_\gamma$ the Dehn twist along a given simple closed curve $\gamma$. In the sequel, we will need Morita's computations for the image of a BSCC map by the second Johnson homomorphism \cite{mor}:
\begin{lemma}
Let $\gamma$ be a BSCC bounding a subsurface $F$ of genus $h$ in $\Sigma$, and let $(u_i,v_i)_{ 1 \leq i \leq h}$ be any symplectic basis of the first homology group of $F$, then we have: \[ \tau_2(T_\gamma) = \left(\sum\limits_{i = 1}^{h} u_i \wedge v_i\right)^{\otimes 2} = \sum\limits_{i = 1}^{h} u_i \odot v_i + \sum\limits_{\substack{ i,j = 1\\ i \neq j}}^{h} \ltree{$u_i$}{$v_i$}{$u_j$}{$v_j$} \in D_2(H). \]
\label{morformula}
\end{lemma}
BSCC maps of genus $1$ and $2$ are all conjugated, by an element of the mapping class group, to one of the Dehn twists $T_{\gamma_1}$ or $T_{\gamma_{1,2}}$ (see Figure \ref{surfcurves} in Section \ref{sec5}). Lemma \ref{morformula} then shows that $\operatorname{Im}(\tau_2)$ is generated by elements of type $u \odot v$ with $\omega(u , v) = 1$ and elements of type $\ltree{$u_1$}{$v_1$}{$u_2$}{$v_2$}$ with $\omega(u_i , v_j) = \delta_{ij}$ and $\omega(u_1 , u_2) = \omega(v_1 , v_2) = 0$. \newline

We also recall that Morita showed in \cite{mor} that the cokernel $D_2(H)/\operatorname{Im}(\tau_2)$ is  a 2-torsion group. Yokomizo showed that whenever $g \geq 2$, its rank over $\mathbb{Z}_2$ is $(g-1)(2g+1)$ \cite{yok}. He gave an explicit basis of the cokernel using the computations of Morita. He also computed that the dimension of $D_2(H)/\tau_2([\mathcal{I},\mathcal{I}])$, which is also a 2-torsion group, is $4g^2 -1$. We shall use the computations of Morita and Yokomizo to prove the second statement in the following theorem. We now suppose that $g \geq 2$.

\begin{theorem}
For any $g \geq 2$, the following homomorphisms

\[
\begin{array}{l}
\left\{ 
\begin{array}{rcl}
D_2'(H) & \overset{\operatorname{Tr}^{sym}}{ \longrightarrow} & \operatorname{Ker}(\omega : S^2(H/2H) \rightarrow \mathbb{Z}_2 ) \\  \ltree{$a$}{$b$}{$c$}{$d$} &  \longmapsto & \omega(a , d) bc + \omega(a , c) bd +\omega(b , d) ac +\omega(b , c) ad \end{array} \right. 
\\
\vspace{0.25pt}
\\
  \left\{ 
\begin{array}{rcl}
  D_2(H) & \overset{\operatorname{Tr}^{as}}{\longrightarrow}  & \operatorname{Ker}(\omega : \Lambda^2(H/2H) \rightarrow \mathbb{Z}_2 ) \\  \ltree{$a$}{$b$}{$c$}{$d$} &  \longmapsto & \omega(a , d) b \wedge c + \omega(a , c) b\wedge d +\omega(b , d) a\wedge c +\omega(b , c) a\wedge d  \\ a \odot b& \longmapsto &(1+\omega(a , b )) a \wedge b
\end{array}
\right.
\end{array}\]
are well-defined, $\operatorname{Sp}(H)$-equivariant, and induce the following commutative diagram with exact rows :

\begin{equation}
\begin{tikzcd} 
0  \arrow[r] & \mathcal{K}/J_3 \arrow[r ,"\tau_2"] & D_2(H) \arrow[r, "\operatorname{Tr}^{as}"] & \operatorname{Ker}(\omega : \Lambda^2(H/2H) \rightarrow \mathbb{Z}_2) \arrow[r] & 0 \\ 0  \arrow[r] & \frac{ [\mathcal{I},\mathcal{I} ]}{J_3 \cap [ \mathcal{I}, \mathcal{I} ] }    \arrow[r, "\tau_2"] \arrow[u, hook] & D'_2(H) \arrow[r , "\operatorname{Tr}^{sym}"] \arrow[u] & \operatorname{Ker}(\omega : S^2(H/2H) \rightarrow \mathbb{Z}_2 ) \arrow[r] \arrow[u]  & 0 
\end{tikzcd}
\label{diag:traces}
\end{equation}

\noindent where the up arrow on the right is induced by the canonical projection $S^2(H/2H)\rightarrow \Lambda^2(H/2H)$.
\label{theoremtraces}
\end{theorem}

\begin{proof}
Let us first show that the maps are well-defined. For $D_2(H)$ we use the presentation from Proposition \ref{prop2.1}, and for $D'_2(H)$ the presentation given by the definition of $\mathcal{A}_2(H)$. It is clear that the antisymmetry relation is sent to 0 since we are working modulo $\mathbb{Z}_2$. Multilinearity is also clear by multilinearity of the symplectic form. Hence, for the tree part, the only relation to check is the $IHX$ relation:

\begin{align*}
 IHX  \longmapsto & ~  \omega(a , d) bc + \omega(a , c) bd +\omega(b , d) ac +\omega(b , c) ad \\ & ~ \omega(d , c) ab + \omega(d , b) ac +\omega(a , c) db +\omega(a , b) dc \\ & ~ \omega(a , d) cb + \omega(a , b) cd +\omega(c , d) ab +\omega(c , b) ad
\end{align*}
which vanishes in $S^2(H/2H)$ and $\Lambda^2(H/2H)$. We have more relations to check for $\operatorname{Tr}^{as}$. The only non-trivial ones are \[  (2(a \odot b)- \ltree{$a$}{$b$}{$a$}{$b$}) \longmapsto 0-2(\omega(a, b) ab) = 0 \] and the one relating halfs of symmetric trees with regular trees (Remark \ref{remtree})
\begin{align*}
(a+b) \odot c - a \odot c - b \odot c \longmapsto &  \quad (1+\omega((a+b) , c))(a \wedge c +b \wedge c) \\ &+ (1+\omega(a, c))a \wedge c \\ &+  (1+\omega(b , c)) b \wedge c
\\  & = \omega(a,c) b \wedge c + \omega(b,c) a \wedge c
\end{align*}
which is also exactly the image of $\ltree{$a$}{$c$}{$b$}{$c$}$.
\newline

It is immediate that $\operatorname{Tr}^{sym}$ and $\operatorname{Tr}^{as}$ are $\operatorname{Sp}(H)$-equivariant, because $\omega$ is, by definition. It is also straightforward to check that they are onto $\operatorname{Ker}(\omega : S^2(H/2H) \rightarrow \mathbb{Z}_2 )$ and $\operatorname{Ker}(\omega : \Lambda^2(H/2H) \rightarrow \mathbb{Z}_2 )$, respectively. Indeed, over $\mathbb{Z}_2$ these kernels respectively have dimensions $ \binom{2g}{2}+2g-1 = (g+1)(2g-1)$ and $\binom{2g}{2}-1 = (g-1)(2g+1)$. We can easily give explicit generators for these spaces and show the desired surjectivity. The elements $a_i b_j$, $a_i a_j$, $b_i b_j$, $a_i a_i$, and $b_i b_i$ (for $1 \leq i,j \leq g$), together with the elements $a_ib_i + a_gb_g$ (for $1 \leq i <g$) are generators for $\operatorname{Ker}(\omega : S^2(H/2H) \rightarrow \mathbb{Z}_2 )$. The projection of these elements in $\Lambda^2(H/2H)$ gives generators for $\operatorname{Ker}(\omega : \Lambda^2(H/2H) \rightarrow \mathbb{Z}_2 )$. To produce elements mapping to one of these generators $cd$ with $\omega(c,d)=0$, we do the following. The genus being greater than or equal to two we can always suppose that there exists $a,b \in H$ with $\omega(a,b)=1$, $\omega(a,c) = \omega(b,d) =0$ and then $\operatorname{Tr}^{sym}\Big(\ltree{$a$}{$d$}{$c$}{$b$}\Big) = cd$. Also $\operatorname{Tr}^{sym}\Big(\ltree{$a_i$}{$a_g$}{$b_i$}{$b_g$}\Big)= a_ib_i + a_gb_g$. The same computations show that $\operatorname{Tr}^{as}$ is onto. 
\newline
Also, we have from \cite{joh2} a set of generators of $\operatorname{Im}(\tau_2)$ which is  sent to $0$ by the map $\operatorname{Tr}^{as}$: for all $(u,v)$ with $\omega(u , v) = 1$ and all $(u_1,v_1,u_2,v_2)$ with $\omega(u_i , v_j) = \delta_{ij}$ and $\omega(u_1 , u_2) = \omega(v_1 , v_2) = 0$ we have \[ \operatorname{Tr}^{as} (u \odot v) = \operatorname{Tr}^{as}\Big( \ltree{$u_1$}{$v_1$}{$u_2$}{$v_2$} \Big) = 0. \]Hence, $\operatorname{Im}(\tau_2)$ is contained in the kernel of $\operatorname{Tr}^{as}$. For the image of $[\mathcal{I},\mathcal{I}]$ by $\tau_2$, it is known that the image is $[\Lambda^3 H,\Lambda^3 H]$ by the surjectivity of $\tau_1$ and the fact that $\tau$ is a Lie algebra homomorphism. Recall that the bracket in $\mathcal{A}^t(H)$ is given by all the ways to contract external vertices using the symplectic form. Taking the bracket of two elements of form $\ltritree{$a$}{$b$}{$c$}$ and $\ltritree{$d$}{$e$}{$f$}$, we get 9 trees, which will be sent by $\operatorname{Tr}^{sym}$ to 36 terms in $S^2(H/2H)$. For example, the coefficient of the symmetric term $ad$ is \[ \omega(b , e ) \omega ( c, f) + \omega(b , f ) \omega ( c, e) + \omega(c , e ) \omega ( b, f) + \omega(c , f ) \omega ( b, e) \]
coming from the trees \[  \ltree{$c$}{$a$}{$f$}{$d$} , \ltree{$c$}{$a$}{$d$}{$e$} , \ltree{$a$}{$b$}{$f$}{$d$} , \ltree{$a$}{$b$}{$d$}{$e$} . \]The above term vanishes, and we thus see that $\tau_2([\mathcal{I}, \mathcal{I}]) \subset \operatorname{Ker}(\operatorname{Tr}^{sym})$.

Finally, the dimensions of the targets of $\operatorname{Tr}^{as}$ and $\operatorname{Tr}^{sym}$ are equal to the ones given by Yokomizo in \cite[Cor.2.2, Cor.3.2]{yok} for the dimensions of the cokernels of $\tau_2$; i.e. $(g-1)(2g+1)$ for $D_2(H)/\operatorname{Im}(\tau_2)$  and $(g+1)(2g-1)$ for $D'_2(H)/(\tau_2([\mathcal{I},\mathcal{I}]))$. This last dimension is not directly given by Yokomizo: it is obtained from the dimension of $D_2(H)/\tau_2([\mathcal{I},\mathcal{I}])$ by removing $\binom{2g}{2}$, because of the exact sequence \eqref{eqlev}.
\end{proof}
Notice that the kernel of the canonical projection $\operatorname{Ker}(\omega : S^2(H/2H) \rightarrow \mathbb{Z}_2) \rightarrow \operatorname{Ker}(\omega : \Lambda^2(H/2H) \rightarrow \mathbb{Z}_2)$ is isomorphic to $H/2H$, which can be mapped into $S^2(H/2H)$ in the obvious way. Hence, applying the snake lemma to the diagram \eqref{diag:traces} and using \eqref{eqlev}, we get the following description of the image of $\mathcal{K}/ [\mathcal{I},\mathcal{I}]$ under $\tau_2$, i.e. the quotient $\mathcal{K}/([\mathcal{I},\mathcal{I}]\cdot J_3$).

\begin{coroll}
There is a short exact sequence: 
\[ 0 \longrightarrow H/2H \longrightarrow \tau_2(\mathcal{K})/ \tau_2([\mathcal{I},\mathcal{I}]) \simeq \mathcal{K}/([\mathcal{I},\mathcal{I}]\cdot J_3) \longrightarrow  \Lambda^2(H/2H) \longrightarrow   0. \]
\label{corinutile}
\end{coroll}

We can relate this short exact sequence to what we know about the abelianization of the Torelli group. For $g \geq 3$, the abelianization of $\mathcal{I}$ is well understood, thanks to the work of Johnson \cite{joh3}. In \cite{johquad}, he built a homomorphism $\beta$ (the so-called Birman-Craggs homomorphism) from the Torelli group to a 2-torsion abelian group $B_{\leq 3}$ (where $B_{\leq k}$ is the filtered space of Boolean polynomial functions of degree at most $k$ on a certain $\mathbb{Z}_2$-affine space), such that the abelianization of $\mathcal{I}$ is isomorphic by $(\tau_1, \beta)$ to a fibered product: \[ \Lambda^3 H \times_{\Lambda^3(H/2H)} B_{\leq 3}.\]
This description implies that $\mathcal{K}/[\mathcal{I},\mathcal{I}]$ is isomorphic to $B_{\leq 2}$ via $\beta$. Johnson also claimed that $\beta(J_3) = B_0$ (see \cite[p.178]{johsurvey}, \cite[Rem. 3.21]{masY3}, and Remark \ref{rembeta} below for a proof). Hence, we have that $J_3 /( [\mathcal{I},\mathcal{I}] \cap J_3)$ is identified to $B_0 \simeq \mathbb{Z}_2$ by the map $\beta$. Therefore, we have $\frac{\mathcal{K}}{[\mathcal{I},\mathcal{I}]\cdot J_3} \stackrel{\beta}{\simeq} B_{\leq 2} / B_0$.  
\noindent Then, the short exact sequence of Corollary \ref{corinutile} fits into the following commutative diagram: 
\begin{equation}
\begin{tikzcd} 
0  \arrow[r]  & H/2H \arrow[r] \arrow[d] &\frac{\mathcal{K}}{[\mathcal{I},\mathcal{I}]\cdot J_3}  \arrow[r] \arrow[d, "\beta"] & \Lambda^2(H/2H) \arrow[r] \arrow[d] & 0 \\ 0  \arrow[r] & B_{\leq 1}/B_{0} \arrow[r] & B_{\leq 2}/B_{0} \arrow[r] & B_{\leq 2}/B_{\leq 1}\arrow[r] & 0.\label{diagyokcor}\end{tikzcd}
\end{equation}

\noindent All vertical arrows are isomorphisms, the left one (respectively the right one) being the inverse of the formal first (respectively second) differential on $B_{\leq 1}$  (respectively $B_{\leq 2}$). We can recover a precise description for the horizontal map $H/2H\rightarrow \frac{\mathcal{K}}{[\mathcal{I},\mathcal{I}]\cdot J_3} $ by investigating in detail the connecting homomorphism arising from the snake lemma applied to diagram \eqref{diag:traces}. The commutativity of the diagram is not trivial and can be deduced from \cite[Prop. 3.3]{yok} or \cite[Lemma 3.18]{masY3}.

\section{Motivations for the study of $\AJ$}
\label{sec3}

We are particularly interested in the relation of the handlebody group with the Johnson filtration. We explain our interest in this filtration and briefly review previous works on this subject.

Below, $\mathcal{V}(3)$ and $\mathcal{S}(3)$ denote respectively the set of all oriented 3-manifolds and all closed oriented homology 3-spheres up to orientation-preserving homeomorphisms. We firstly remind some facts about Heegaard splittings. Any 3-manifold can be divided (not in a unique way) in two handlebodies of same genus. Equivalently, any 3-manifold can be obtained by gluing two handlebodies together by a homeomorphism between their boundaries. Essentially, this homeomorphism specify where a set of meridians of the second handlebody should be sent on the boundary of the first one, yielding the notion of Heegaard diagrams. 

The standard exemple is of course the sphere $S^3$, where one considers the standard handlebody $V_g$ and glues a copy $-V_g$ with opposite orientation by a map sending its meridians to the curves $\beta_i$ in Figure \ref{surface}. Then we get for all $g$ a splitting $S^3 := {V_g} \underset{\iota_g}{\cup} (-V_g)$, where $\iota_g$ is a certain orientation-preserving homeomorphism of $\Sigma_g$ which can be defined by giving its action on $\pi$ (see Section \ref{sec6}). Note that there is, up to isotopy, a unique Heegaard splitting of $S^3$ of genus $g$. We define $\mathcal{B}_{g,1} :=  \iota_g\mathcal{A}_{g,1}\iota^{-1}_g$. We denote by $S^3_\varphi$ the 3-manifold $V_g \underset{\iota \circ \varphi}{\cup} (-V_g)$ for any element $\varphi \in \mathcal{M}_{g,1}$(we extend $\varphi$ to $\Sigma_g$ by the identity on the remaining disk). The map $\varphi$ is called the \emph{gluing map}. We also have stabilization maps $\mathcal{M}_{g,1} \rightarrow \mathcal{M}_{g+1,1} $, compatible with the other maps. When one composes the gluing map on the right, by an element of $\mathcal{B} = \mathcal{B}_{g,1}$ or to the left by an element of $\mathcal{A} $, the resulting manifold does not change up to homeomorphism. The following result is a refinement of the Reidemeister-Singer theorem \cite{rei, sin}.

\begin{theorem}[Reidemeister-Singer]
There is a bijection  \begin{align*}
\underset{g \rightarrow + \infty}{\operatorname{lim}} \mathcal{A}_{g,1} \setminus \mathcal{M}_{g,1} / \mathcal{B}_{g,1} & \longrightarrow \mathcal{V}(3) \\ \varphi  & \longmapsto S^3_\varphi
\end{align*}
which actually restricts to a bijection $\underset{g \rightarrow + \infty}{\operatorname{lim}}  \mathcal{A}_{g,1} \setminus \mathcal{I}_{g,1} / \mathcal{B}_{g,1} \longrightarrow \mathcal{S}(3)$.
\label{rsing}
\end{theorem}
\noindent  The second fact in Theorem \ref{rsing} is written explicitly in \cite{mat87}. One would expect that considering restrictions to deeper groups of the Johnson filtration would yield other topological conditions on the manifold, but this is not the case in low degrees for homology 3-spheres. We call a homology 3-sphere \emph{$J_k$-equivalent} to $S^3$ if it homeomorphic to $S^3_\varphi$ for some $\varphi$ in $J_k$. More generally, we say that two 3-manifolds are \emph{$J_k$-equivalent} if there exists a Heegaard splitting of the first one such that one can compose the gluing map by an element of $J_k$ and get a Heegaard presentation for the second manifold. It is known that $J_k$-equivalence is an equivalence relation.

Morita \cite{mor} has shown that any two homology 3-spheres are $J_2$-equivalent. Pitsch \cite{pit} improved this result to $J_3$-equivalence. They both used the following.

\begin{lemma}
Let $l \geq 1$. If for some genus $g$, we have $ \operatorname{Im}(\tau_k) = \tau_k(\AJk)+ \tau_k(\BJk)$ for all $k \leq l$ then any homology 3-sphere is $J_{l+1}$-equivalent to $S^3$.
\label{pitschlemma}
\end{lemma}

Another proof of the fact that any two homology 3-spheres are $J_3$-equivalent is given in \cite[Theorem C]{masY3}. Unfortunately, using Lemma \ref{pitschlemma} for $l=3$ seems complicated: the computations could hardly be made by hand, and we do not know how to build all elements of $J_3$, whereas $J_2$ has well-known generators. Besides, as the result involves $\mathcal{B}$, this lemma only addresses the question of homology 3-spheres, hence manifolds at least $J_1$-equivalent to $S^3$. That is one reason why we want to describe in this paper $\tau_2(\AJ)$ by polarizing some computations in \cite{pit} and by introducing new arguments.

We also know some facts about the first term $\mathcal{I}\mathcal{A} := \mathcal{A}\cap J_1$. A generating set was described by Omori in \cite{omo}. He gives the following theorem, where HBP stands for  ``homotopical bounding pair'', and a genus-$h$ HBP-map is the composition $T_c \circ T_d^{-1}$ of two Dehn twists where $c$ and $d$ are essential simple closed curves cobounding a surface of genus $h$, cobounding an annulus in the handlebody, and not bounding disks in the handlebody.  

\begin{theorem}[Omori]
For $g \geq 3$,  $\mathcal{I}\mathcal{A}_{g,1}$ is normally generated in $\mathcal{A}_{g,1}$ by a genus-1 HBP-map, and hence it is generated by genus-1 HBP-maps.
\label{omori}
\end{theorem}

\noindent It would be interesting to get the same kind of description for $\AJ$, but we only give in this paper its image by the second Johnson homomorphism, and formulate some questions (see Remark \ref{finalremark}).

But our main motivation for the study of $\AJ$ comes from \cite{lev}. In this paper, Levine defines the \emph{Lagrangian filtration} $(L_k)_{k \geq 1}$ which is a non-separating filtration of the mapping class group. It is not helpful to get an approximation of the mapping class group of the surface, but it is natural to study 3-manifolds presented through Heegaard splittings. The definition of this filtration depends on $\mathbb{A}$, the kernel of the projection $p$ from $\pi$ to $\pi'\simeq  \pi / \mathbb{A}$. The Lagrangian subgroup $A$ is the kernel of the projection $H \rightarrow H'$ which is the image of $\mathbb{A}$ under the projection from $\pi$ to $H$. Also, whenever $f$ is an element of the mapping class group, $f_* \in \operatorname{Sp}(H)$ stands for the action of $f$ on $H$. We still write abusively $f$ for the action of $f$ on the fundamental group.

\begin{definition}
The Lagrangian Torelli group is defined by \[ \mathcal{I}^L := \{ h \in \mathcal{M} \mid h_{*}(A) \subset A \text{ and } h_{*} \text{ is the identity on } A\}. \] For $k \geq 1$, an element $h$ of $\mathcal{M}$ belongs to $L_k$ if it is in $\mathcal{I}^L$ and $p(h((\mathbb{A})) \subset \Gamma_{k+1}\pi'$.
\label{deflt}
\end{definition}

\noindent Note that $L_1 = \mathcal{I}^L$. We remind the following fact from \cite{lev}, describing the intersection of this filtration, which is non-empty.

\begin{lemma}
$L_\infty := \underset{k \geq 1}{\bigcap} L_k$ coincides with $\mathcal{A} \cap L_1$.
\end{lemma}

\noindent It is clear that $J_k \subset L_k$ for all $k \geq 1$. 

\begin{question}
Do we have $L_k = J_k \cdot L_\infty$ for all $k$ ?
\label{levquestion}
\end{question}

This question can be approached inductively, which leads to the next lemma, given by Levine (see \cite[Lemma 6.2]{lev} for a proof).

\begin{lemma}
Suppose $L_k = J_k \cdot L_\infty$, then $L_{k+1} = J_{k+1} \cdot L_\infty$ if and only if \[ \operatorname{Im}(\tau_k) \cap  \operatorname{Ker}(D_k(H) \rightarrow D_k(H')) = \tau_k(\mathcal{A} \cap J_k).\]
\label{lemma37}
\end{lemma}
\noindent It is shown in \cite[Lemma 6.3]{lev} that $L_1 = J_1 \cdot L_\infty$. Furthermore, the following proposition, describing $\frac{\mathcal{A}\cap J_1}{\AJ}$, was given by Morita in \cite[Lemma 2.5]{mor}: 
\begin{propo}
We have $ \operatorname{Ker}(D_1(H) \rightarrow D_1(H')) = \tau_1(\mathcal{A} \cap J_1)$.
\label{morprop}\end{propo}
\noindent Recall that $\tau_1$ is surjective, hence this proposition together with Lemma \ref{lemma37} implies that the answer to Levine's question is positive for $k= 1,2$ (as explained in \cite[Proposition 6.1]{lev}). As for the $k=3$ case, the equality necessary for the induction step is no longer true, as will be shown in the next section: \[ \tau_2(\mathcal{A} \cap J_2) \subsetneq  \operatorname{Im}(\tau_2) \cap  \operatorname{Ker}(D_2(H) \rightarrow D_2(H')).\] Therefore the answer to Question \ref{levquestion} is ``no'' for $k=3$.
\section{The A-trace}
\label{sec4}
In this section, we are still working with two ``abstract" surfaces $\Sigma_{g,1} \subset \Sigma_g$ bounding a handlebody: $\Sigma_g = \partial V_g$. We consider the subgroup $\mathcal{A}$ of $\mathcal{M}$ consisting of elements of the mapping class group of $\Sigma$ extending to $V$. The context differs from \cite{pit}, where there are two handlebodies defined by a Heegaard splitting of $S^3$. In this paper, we wish to investigate about $\tau_2(\AJ)$. Considering that an element of $\mathcal{A}$ globally preserves $\mathbb{A}$, it is not hard to see that the $k$-th Johnson homomorphism sends an element of $\mathcal{A} \cap J_k$ to the sum of an element in $A \otimes \mathcal{L}_{k+1}(H)$ and an element in $H \otimes  \operatorname{Ker} (\mathcal{L}_{k+1}(H) \rightarrow \mathcal{L}_{k+1}(H'))$. Hence we certainly have $\tau_k(\mathcal{A}\cap J_k) \subset \tau_k(J_k) \cap \operatorname{Ker}(D_k(H) \rightarrow D_k(H'))$. It is not easy to see what could be another necessary condition to be in $\tau_k(\AJk )$ . Hence one could wonder, in relation to Question \ref{levquestion}, whether $\tau_2 (\AJ)$ coincides with $ \operatorname{Im}(\tau_2) \cap  \operatorname{Ker}(D_2(H) \rightarrow D_2(H'))$. We show in this section that it is not the case.

\subsection{Examples of elements of $\AJ$}

Here, we describe three families of examples of elements in $\AJ$. We start by recalling some facts about the generation of $\mathcal{A}$ and $J_2$.

First, a Dehn twist along a simple closed curve belongs to the handlebody group if and only if this curve bounds a disk in the handlebody $V$. Such a meridional twist can also be performed half-way. Furthermore, if two curves $\delta$ and $\delta '$ cobound a properly embedded annulus in $V$, one can perform an annulus twist in the handlebody and see that $T_{\delta} T_{\delta '}^{-1}$ is an element of $\mathcal{A}$. The handlebody group is generated by meridional twists, meridional half-twists and annulus twists. See \cite{hensel} for more details. 

As for the second term in the Johnson filtration, it is generated by BSCC maps \cite{joh2}, i.e. Dehn twists along simple closed curves bounding in the surface. Also, it is a classical fact from \cite{mor91} that $[J_k, J_l] \subset J_{k+l}$, so any commutator of two elements of the Torelli group are in the Johnson subgroup $\mathcal{K} = J_2$.

Knowing these facts we can build three families of elements in $\AJ$: \begin{enumerate}
	\item Dehn twists along bounding simple closed curves, which also bound disks in the handlebody.
	\item Annulus twists along two simple closed curves which are both bounding subsurfaces in the surface (but not necessarily bounding disks in the handlebody).
	\item Commutators of the group $\mathcal{A}\cap J_1$, the Torelli handlebody group, for which a generating system is recalled in Theorem \ref{omori}.
\end{enumerate}


We shall now define a map $ \operatorname{Tr}^A :  \operatorname{Ker}(D_2(H) \rightarrow D_2(H')) \rightarrow S^2(H')$, and show that it vanishes on all the image of $\AJ$ under $\tau_2$.

\label{section51}
\subsection{Definition of the A-trace}

We consider the following filtration on $D_k(H)$, which only depends on the Lagrangian subgroup $A$ of $H$. For $-1 \leq l \leq k+1$, we set: 
\[ \mathcal{F}_l = \operatorname{Span} \bigg < \begin{matrix}
\text{expansion of trees with $k$ nodes (and halfs of symmetric trees when $k$ is even)} \\ \text{ with at least}  ~ l+1 ~ \text{leaves vanishing in } H'
\end{matrix}  \bigg >.\]
\noindent Below, for $k=2$, we identify trees and their expansions (see Remark \ref{remtree}).

We consider the following diagram, where all vertical arrows are induced by the projection from $H$ to $H'$: 
\[
\begin{tikzcd}
0 \arrow[r] & A \otimes \mathcal{L}_{k+1}(H) \arrow[r] \arrow[d, two heads] & H \otimes \mathcal{L}_{k+1}(H) \arrow[r] \arrow[d, two heads, "p"] & H' \otimes \mathcal{L}_{k+1}(H) \arrow[r] \arrow[d, two heads] & 0 \\
0 \arrow[r] & A \otimes \mathcal{L}_{k+1}(H')\arrow[r] & H \otimes \mathcal{L}_{k+1}(H') \arrow[r]  & H' \otimes \mathcal{L}_{k+1}(H') \arrow[r]  & 0.
\end{tikzcd}\]

\noindent We claim that the following holds: 

\begin{lemma}Set $K:=  \operatorname{Ker}( \mathcal{L}_{k+1}(H) \rightarrow \mathcal{L}_{k+1}(H'))$. We have: 
\begin{enumerate}[label=\normalfont(\roman*)]
\item $\mathcal{F}_{-1} = D_k(H)$
\item $\mathcal{F}_{0} \subset D_k(H) \cap p^{-1}(A \otimes \mathcal{L}_{k+1}(H')) =  \operatorname{Ker}(D_k(H) \rightarrow D_k(H'))$
\item $\mathcal{F}_{1} \subset D_k(H) \cap  \operatorname{Ker}(p) = D_k(H) \cap (H \otimes K)$
\end{enumerate}
\label{lemmainclusions}
\end{lemma}

\begin{proof}
We have seen in Section \ref{sec2} that expansions of half symmetric trees and expansions of trees lie in $D(H)$. If a tree with $k$ leaves has at least one leaf in $A$, then after expanding the tree, there will be $k-1$ terms in which such leaf is involved in the free Lie algebra part. This $k-1$ terms will vanish after projecting on $\mathcal{L}(H')$. The remaining term will be a tensor product of the root vanishing in $H'$ and some element in $\mathcal{L}(H)$. This shows that $p(\mathcal{F}_0) \subset A \otimes \mathcal{L}_{k+1}(H')$. If the tree has at least two leaves in $A$, then the expansion gives $k$ terms such that the part in the free Lie algebra vanish in $\mathcal{L}(H')$. Hence $p(\mathcal{F}_1)= {0}$.
\end{proof}

\begin{rem}
In fact all the inclusions in Lemma \ref{lemmainclusions} are equalities, but we shall not need this.
\end{rem}
\begin{rem}
The graded space associated with the filtration $(\mathcal{F}_l)_{-1 \leq l \leq k+1} $ can be identified to the space $\mathcal{A}^t_k(A \oplus H')$ of tree-like Jacobi diagrams colored by $A\oplus H'$ with degree defined by the number of $A$-colored leaves shifted by 1 (the same space appears with a different grading in \cite{ver}).
\label{remvera} 
\end{rem}
Besides, the long exact sequence in relative homology for the handlebody
\[
\begin{tikzcd}
0 \arrow[r] & H_2(V, \partial V; \mathbb{Z}) \arrow[r] &  H \simeq H_1( \partial V; \mathbb{Z}) \arrow[r]  & H' = H_1( V; \mathbb{Z}) \arrow[r]  & 0
 \end{tikzcd}\]
 
\noindent gives a canonical isomorphism $H_2(V, \partial V; \mathbb{Z}) \simeq A$. Now, Poincaré-Lefschetz duality $$H_2(V,\partial V;\mathbb{Z})\simeq~H^1(V; \mathbb{Z})\simeq~(H')^*$$ gives an intersection pairing \[
\omega' : A \otimes H' \longrightarrow \mathbb{Z} \] \noindent which is also induced by $\omega$ in the obvious way.
Then, by considering the injection $i$ of $\mathcal{L}(H')$ in the tensor algebra $T(H')$, and the contraction $({\omega'})^{1,2}$ of the first two tensors in $A \otimes T(H')$ by $\omega'$, we define the following map: 
\[
\begin{tikzcd}
 \operatorname{Tr}^A : \mathcal{F}_0 \arrow[r, "p"] &  A\otimes \mathcal{L}_{k+1}(H') \arrow[r, "i"]  & A \otimes T_{k+1}(H') \arrow[r, "(\omega')^{1,2}"]  &  T_{k}(H') \arrow[r, two heads] &  S^k(H').
 \end{tikzcd}\]
 
\begin{rem}
The definition of the homomorphism $\operatorname{Tr}^A$ is inspired by the trace $\operatorname{Tr}$ defined by Morita in \cite{mor93}, but the reader should be aware that the following diagram  does not commute: \[ \begin{tikzcd} \mathcal{F}_0 \arrow[r,hook] \arrow[d, "\operatorname{Tr}^A"] & D_k(H) \arrow[d, "\operatorname{Tr}"] \\ S^k(H') & S^k(H). \arrow[l, two heads] \end{tikzcd}\]
\end{rem}

The first thing to notice about $\operatorname{Tr}^A$ is that it vanishes on $\mathcal{F}_1$ as $p$ already vanishes on this space. Hence it can be thought of as a map starting from $\mathcal{F}_0 / \mathcal{F}_1$. Therefore, to compute this map, we can consider only trees with one leaf colored by $A$ and the other leaves non-trivial in $H'$. The map $\operatorname{Tr}^A$ is thus defined on the graded space associated with the filtration $\mathcal{F}$, which corresponds to diagrams whose leaves are colored by $A$ or $H'$ (see Remark \ref{remvera}). On such a space, a direct computation shows that there is a practical way of computing $ \operatorname{Tr}^A$: take the leaf colored by $A$ and consider all possible ways to contract it by $\omega'$ with the other leaves in $H'$. One gets a sum of oriented circles with leaves in $H'$ (the orientation being given by drawing an arrow from the leaf in $A$ to the other leaf).  One can read this oriented diagram in $S^k(H')$, the inward leaves contributing with a minus sign. We now denote by $x'$ the class in $H'$ of an element $x$ in $H$. We will also omit some tensor product notations when it is clear from context.

\begin{example}
For $a\in A$ and $c,d,e \in H$ we have:

\begin{equation}
\operatorname{Tr}^A\Big(\ltree{$a$}{$c '$}{$d'$}{$e'$}\Big) = \omega(a,e) d'c' - \omega(a,d) e'c'.
\label{eqexemple}
\end{equation}

\noindent Indeed, in $S^2(H')$, we have: 
\begin{align*}
(\omega')^{1,2} \circ i \circ p \Big(\ltree{$a$}{$c$}{$d$}{$e$}\Big) &= (\omega')^{1,2} \circ i (a \otimes [[e',d'],c'])  \\ &=  (\omega')^{1,2}(ae'd'c' -ad'e'c' - ac'e'd' + ac'd'e')\\ &= \omega(a,c)(d'e' -e'd') - \omega(a,d)e'c' + \omega(a,e)d'c' \\ &= \omega(a,e) d'c' - \omega(a,d) e'c' \quad  \in S^2(H').
\end{align*}
\label{ex43}
\end{example}

\begin{rem}
It is worth noting that the restriction of the Johnson filtration to $\A$ is compatible with the conjugation by elements of $\A$. This induces a $\rho_0(\A)$-module structure on the quotients $\A \cap J_k / \A \cap J_{k+1}$, where $\rho_0(\A)$ is the image of $\A$ under the representation $\rho_0 : \mathcal{M} \rightarrow \operatorname{Sp}(H)\subset \operatorname{Aut}(H)$. The action of $\rho_0(\A)$ on $H$ induces an action on $H'$ (and thus on $S^k(H')$). The map $\operatorname{Tr}^A$ is equivariant relatively to these actions.
\label{reminvariance}
\end{rem}

We now focus on the case $k=2$. One could check by direct computation that this map vanishes on the image by $\tau_2$ of  all elements of $ \AJ$ of the three kinds described in Section \ref{section51}. Instead of that, we will show in the next section that the map actually vanishes on the whole of $\tau_2(\mathcal{A} \cap J_2)$. Nevertheless, this map is not trivial on $\mathcal{F}_0 \cap \operatorname{Im}(\tau_2)$, as we shall now see. We fix a choice $(a_i, b_i)_{1 \leq i \leq g}$ of a symplectic basis for $H$, such that $A$ is generated by the family $(a_i)_{1 \leq i \leq g}$. For instance, we consider the basis of $H$ induced by a system of meridians and parallels $(\alpha_i, \beta_i)_{1 \leq i \leq g}$ as explained in Section \ref{sec1}. Let us define two families of elements in $\mathcal{F}_0$, depending of the previous choice: 

\begin{align*}
T_1^{ij} & := \ltree{$a_i$}{$b_j$}{$b_j$}{$b_i$} \quad i \neq j, \\
T_2^{kk',ij} & := \ltree{$a_k$}{$b_i$}{$b_j$}{$b_k$} + \ltree{$a_{k'}$}{$b_i$}{$b_j$}{$b_{k'}$} \quad i \neq j, k \neq j, k' \neq j.
\end{align*}

\begin{lemma}
The elements $T_1^{ij}$ and $T_2^{kk',ij}$ belong to $ \operatorname{Im}(\tau_2)$ and \begin{align*}
 \operatorname{Tr}^A (T_1^{ij}) &= b_j'b_j' \\   \operatorname{Tr}^A(T_2^{kk',ij}) &= 2b_i'b_j'.
\end{align*}
\label{cextrees}
\end{lemma}

\begin{proof}
By definition of $\operatorname{Tr}^{as}$, we get $ \operatorname{Tr}^{as}(T_1^{ij}) = \omega(a_i,b_i)b_j \wedge b_j = 0 \in \Lambda^2(H/2H)$ and $ \operatorname{Tr}^{as}(T_2^{kk',ij}) = 2\omega(a_k,b_k)b_i \wedge b_j = 0 \in \Lambda^2(H/2H)$. Therefore, by Theorem \ref{theoremtraces},we have $T_1^{ij}, T_2^{kk',ij} \in \operatorname{Im}(\tau_2)$. For the computation of $\operatorname{Tr}^A$ on $T_1^{ij}$ and $T_2^{kk',ij}$, we use formula \eqref{eqexemple}.
\end{proof}
We embed $S^2(H')$ in $(H' \otimes H')^{\mathfrak{S}_2}$ by sending $h_1'h_2' \in S^2(H')$ to $h_1' \otimes h_2' + h_2' \otimes h_1' \in (H' \otimes H')^{\mathfrak{S}_2}$. It defines a restriction map $(H' \otimes H')^* \rightarrow S^2(H')^*$. Using the duality $A \simeq H'^*$ given by the map $a \mapsto \omega' (a,-)$, we then have an isomorphism from $(A \otimes A)^{\mathfrak{S}_2}$ to $(H'^* \otimes H'^*)^{\mathfrak{S}_2}$ which is a subspace of $(H'^* \otimes H'^*) \simeq (H' \otimes H')^*$. Hence we obtain a well-defined map $r$ from $(A \otimes A)^{\mathfrak{S}_2}$ to $S^2(H')^*$:
\begin{equation}
r: (A \otimes A)^{\mathfrak{S}_2} \longrightarrow (H'^* \otimes H'^*)^{\mathfrak{S}_2} \longrightarrow (H'^* \otimes H'^*) \simeq (H' \otimes H')^*\longrightarrow S^2(H')^*
\label{req}
\end{equation}
 Notice that $r(a_i\otimes a_i) = 2(b_i'b_i')^*$ and $r(a_i \leftrightarrow a_j) = 2(b_i'b_j')^*$, which shows that $r/2$ is well-defined, surjective, and hence is an isomorphism. We can now define $ \widetilde{\operatorname{Tr}}^A$ as the bilinear map
\[
\begin{array}{rcl}
\widetilde{\operatorname{Tr}}^A : \mathcal{F}_0 \times (A \otimes A)^{\mathfrak{S}_2} & \longrightarrow & \mathbb{Z}\\
(T,s) &\longmapsto & \frac{1}{2}r(s)(\operatorname{Tr}^A(T)).
\end{array}\]

\noindent We can also regard $\widetilde{\operatorname{Tr}}^A$ as a bilinear map: 
$\mathcal{F}_0/\mathcal{F}_1 \times (A \otimes A)^{\mathfrak{S}_2} \rightarrow  \mathbb{Z}.$

\begin{rem}
Notice that $\operatorname{Tr}^A$ depends only on the choice of the Lagrangian subgroup $A \subset H$.
\end{rem}
\begin{rem}
Since $r/2$ is an isomorphism, for any $T \in \mathcal{F}_0$, we have that $\widetilde{\operatorname{Tr}}^A(T,s) = 0$ for all $s \in (A \otimes A)^{\mathfrak{S}_2}$ if and only if $\operatorname{Tr}^A(T)= 0$.
\label{newrem}
\end{rem}

\label{sect52}
\subsection{Relating $ \operatorname{Tr}^A$ with the Casson invariant}

In this section, we review Morita's decomposition of the Casson invariant in \cite{mor} and use it to show the following: 

\begin{theorem}
The Casson invariant induces a map $\mu : D_2(H) \times \mathcal{M} \rightarrow \mathbb{Z}$, which is not bilinear. Its restriction to $\mathcal{F}_0 \times  \mathcal{I}^L$ is bilinear and fits into a commutative diagram \[
\begin{tikzcd}
\mathcal{F}_0  ~ \times ~ \mathcal{I}^L \arrow[r, "\mu"] \arrow[d, two heads, xshift =2.5ex, , "\sigma"] \arrow[d, two heads,xshift = -3ex] &\mathbb{Z} \\
\mathcal{F}_0/\mathcal{F}_1  \times (A \otimes A)^{\mathfrak{S}_2} \arrow[ru, start anchor = east, " \widetilde{\operatorname{Tr}}^A", bend right = 25] &
\end{tikzcd}.\] \noindent Furthermore, for any  $T \in \tau_2(\mathcal{A} \cap J_2)$ and any $\varphi \in \mathcal{I}^L$, $\mu(T,\varphi) = 0$. Consequently, $ \operatorname{Tr}^A$ vanishes on $\tau_2(\mathcal{A} \cap J_2)$.
\label{theoremtrA}
\end{theorem}
The map $\sigma$ is defined in the following way. Recall from Definition \ref{deflt} that $\mathcal{I}^L$ is the Lagrangian Torelli group. For $f \in \mathcal{I}^L$ and $h \in H$, the difference $f_*(h)-h$ only depends on the class of $h$ in $H'$, and is in $A$ because of the very definition of $\mathcal{I}^L$. Hence we get a map
\[ 
\mathcal{I}^L \longrightarrow  \operatorname{Hom}(H',A) \simeq (H')^* \otimes A \simeq A \otimes A \]whose target restricts to $(A \otimes A)^{\mathfrak{S}_2}$ because of the symplectic condition. Hence we get a homomorphism $\sigma : \mathcal{I}^L \rightarrow (A \otimes A)^{\mathfrak{S}_2}$.
Let us describe $\sigma$ in terms of the symplectic basis described in Section \ref{sect52}. It is known that the canonical map from $\mathcal{M}$ to $\operatorname{Sp}(H)$ given by the action in homology is surjective. Using the symplectic basis, we identify $\operatorname{Sp}(H)$ with the group $\operatorname{Sp}(2g,\mathbb{Z})$ of matrices $M$ such that $  M^TJM = J$ where $J:=\begin{pmatrix}
0 & Id\\
-Id & 0
\end{pmatrix}$, i.e. matrices $M =\begin{pmatrix}
A & B\\
C & D
\end{pmatrix}$ where $A,B,C$ and $D$ satisfy the following equations:

\begin{align}
\notag A^TD - C^TB &= Id \\ 
A^TC &= C^TA \\
\notag D^TB &= B^TD.
\label{eqsympl}
\end{align}

\noindent The image of $\mathcal{A}$ by $\mathcal{M}\rightarrow \operatorname{Sp}(2g,\mathbb{Z})$  consists of all matrices of the form $\begin{pmatrix}
A & B\\
0 & D
\end{pmatrix}$ where $A^TD = Id$ and $D^TB$ is symmetric (see \cite[Lemma 2.2]{birman75} or \cite{hirose}). The image of $\mathcal{I}^L$ by $\mathcal{M}\rightarrow \operatorname{Sp}(2g,\mathbb{Z})$  consists of all  matrices of type $\begin{pmatrix}
Id & S\\
0 & Id
\end{pmatrix}$ where $S$ is symmetric. The matrix $S$ associated in this way to an element $\varphi$ is the description of $\sigma(\varphi) \in \operatorname{Hom}(H',A)$ in the basis $(b_i')_{1  \leq i \leq g}$ and $(a_i)_{1  \leq i \leq g}$. In particular, $\sigma$ is surjective. The matrix $S = (S_{i,j})_{1 \leq i,j \leq g}$ actually corresponds to the symmetric tensor $\sum_{i,j=1}^g S_{i,j}(a_i \otimes a_j) \in (A \otimes A)^{\mathfrak{S}_2}$ (via the isomorphism $(A \otimes A)\simeq \operatorname{Hom}(H',A)$ given by $\omega'$).

\begin{rem}The map $\sigma$ can even be restricted to $\mathcal{I}^L \cap \mathcal{A}$, and will still be onto $(A \otimes A)^{\mathfrak{S}_2}$ (as a consequence of \cite[Theorem 7.1]{hensel}). This has a role to play in the proof of Theorem \ref{theoremtrA}.
\label{remI}
\end{rem}
The following corollary is a consequence of Theorem \ref{theoremtrA}.

\begin{coroll}
For any $g \geq 2$, $\tau_2(\mathcal{A} \cap J_2)$ is strictly included in $  \operatorname{Im}(\tau_2) \cap  \operatorname{Ker}(D_2(H) \rightarrow D_2(H'))$.
\end{coroll}

\begin{proof}
We have exhibited in Lemma \ref{cextrees} elements of $  \operatorname{Im}(\tau_2) \cap  \operatorname{Ker}(D_2(H) \rightarrow D_2(H'))$ on which $ \operatorname{Tr}^A$ does not vanish.
\end{proof}

The rest of this section is dedicated to the proof of Theorem \ref{theoremtrA}, and in particular to the construction of $\mu$.

\subsubsection{Morita's decomposition of the Casson invariant}
Let $\lambda$ denote the Casson invariant. We consider a Heegaard embedding $j : \Sigma_{g,1} \rightarrow S^3$ of our abstract surface $\Sigma_{g,1}$ in $S^3$. This means that there exists a surface $\overline{\Sigma_g} \subset S^3$ such that $\overline{\Sigma_{g,1}} := j(\Sigma_{g,1})$ is obtained from $\overline{\Sigma_g}$ by removing a small open disk, and such that $\overline{\Sigma_g}$ splits $S^3$ in two handlebodies $\overline{V_g}$ and $\overline{W_g}$, which are called the ``inner" and the ``outer" handlebody, respectively. The orientation that $j$ induces on $\overline{\Sigma_{g,1}}$ is supposed to coincide with the one induced by $\overline{V_g}$. Later, we will also suppose that $j$ extends to $V_{g}$, and that $j(V_{g})$ is the ``inner" handlebody $\overline{V_g}$ in the splitting of $S^3$. Then, the handlebody group $\mathcal{A} = \mathcal{A}_{g,1}$ is identified through $j$ to the mapping class group of $\overline{V_g}$ relative to the disk $\overline{\Sigma_g} \smallsetminus \overline{\Sigma_{g,1}}$.

For every $\varphi \in \mathcal{I}$, one can define the 3-manifold obtained by cutting $S^3$ along the image of $j$ and gluing back the two handlebodies using the mapping cylinder of $\varphi$. In \cite{mor}, Morita defines $\lambda_j (\varphi)$ as  the Casson invariant of the resulting homology 3-sphere $S^3(j,\varphi)$, yielding a map: \begin{align*}
\lambda_j : \mathcal{I} &\longrightarrow \mathbb{Z} \\
\varphi & \longmapsto \lambda(S^3(j,\varphi)).
\end{align*}

The above map is \emph{not} a homomorphism, nevertheless Morita showed that its restriction to $\mathcal{K} =J_2$ is a homomorphism. He also showed that it can be expressed as the sum of two homomorphisms. We review their definitions, and refer the reader to \cite{mor} or \cite{masY3} for more details. The first one, $d$, is called the ``core of the Casson invariant'' and is independent of $j$. The second one is not, but is completely determined by the second Johnson homomorphism. Our notation conventions differ slightly from the original ones given in \cite{mor}, the content being exactly the same.

We do not need to give a precise definition for the map $d :\mathcal{K} \rightarrow \mathbb{Z}$, we only need to recall the following facts. Johnson showed $\cite{joh2}$ that $\mathcal{K}$ is generated by Dehn twists along bounding simple closed curves and Morita proved in \cite{mor} that \[ d(T_\gamma) = 4h(h-1)\] whenever $\gamma$ is a simple closed curve bounding a subsurface of genus $h$.

As for the second map, we need to fully review its definition. Let $\mathcal{C}$ be the unital, commutative, and associative algebra with generators $l(u,v)$ for all $u$ and $v$ in $H$ and subject to the relations: 
\begin{align*}
l\left(n \cdot u+n^{\prime} \cdot u^{\prime}, v\right) & =n \cdot l(u, v)+n^{\prime} \cdot l\left(u^{\prime}, v\right) \\
l(v, u) &=l(u, v)+\omega(u, v),
\end{align*}

\noindent for all $u,u',v$ $\in$ $H$ and for all $n,n' \in \mathbb{Z}$. We denote by $\operatorname{lk}$ the linking number in $S^3$. Let $\varepsilon_j : \mathcal{C} \rightarrow \mathbb{Z}$ be the unique algebra homomorphism such that: \[\varepsilon_{j}(l(u, v)):=\operatorname{lk}\left(j_*(u), j^+_*(v)) \right) \]

\noindent where $j^+$ is an embedding parallel to $j$, meaning that the image of $j^+$ is obtained by pushing the image of $j$ towards the outer handlebody. We fix a set of meridians and parallels $(\overline{\alpha},\overline{\beta})$ for the surface $\overline{\Sigma_{g,1}}$ (see Figure \ref{handlebody}). This defines a system $(\alpha, \beta)$ of meridians and parallels for $\Sigma$ given by $\alpha := j^{-1}(\overline{\alpha})$ and $\beta := j^{-1} (\overline{\beta})$. For any $1  \leq i \leq g$, the homology classes of $\alpha_i$ and $\beta_i$ are denoted respectively by $a_i$ and $b_i$.

\begin{figure}[h]
	\centering
	\includegraphics[scale= 0.23]{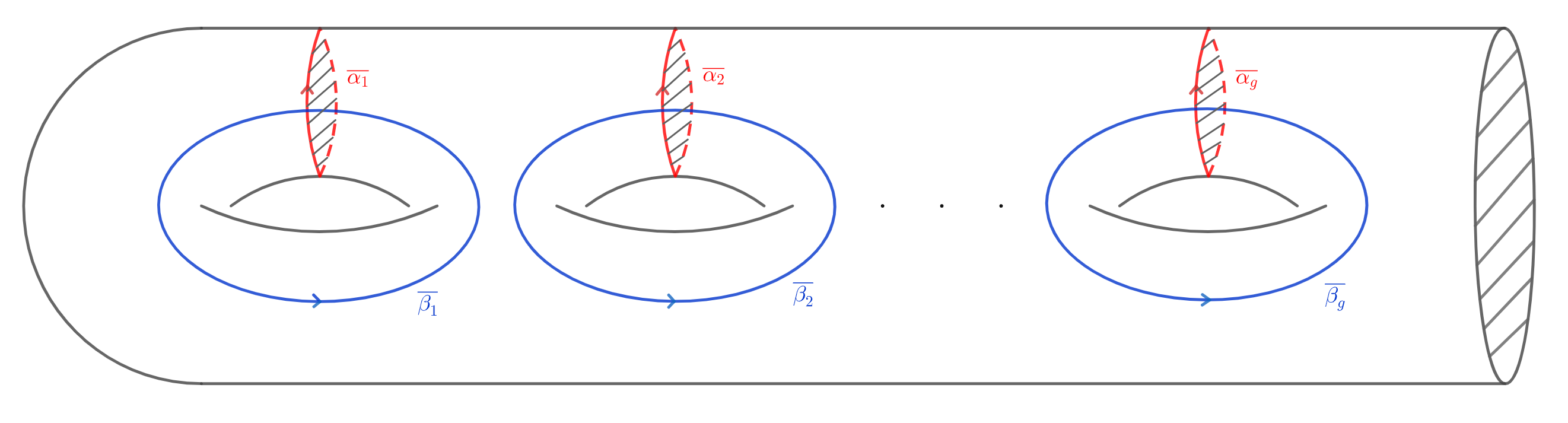}
	\caption{A system of meridians and parallels on $\overline{\Sigma_{g,1}} \subset \overline{V_g} \subset S^3$.}
	\label{handlebody}
\end{figure}

\begin{rem}
Considering that $\operatorname{lk}(j_*(a_i), j_*^+(b_j)) = 0$ and $\operatorname{lk}(j_*(b_i), j^+_*(a_j)) = \delta_{ij}$, the matrix associated to the bilinear mapping $\operatorname{lk}(j_*(-),j^+_*(-)) : H \times H \rightarrow \mathbb{Z}$ is $\begin{pmatrix}
0 & 0 \\ Id & 0
\end{pmatrix}$.
\end{rem}

Morita also defines a map $\theta:\left(\Lambda^{2} H \otimes \Lambda^{2} H\right)^{\mathfrak{S}_{2}} \rightarrow \mathcal{C}$ determined by: \begin{align*}
\theta((u \wedge v) \otimes(u \wedge v)) & :=l(u, u) l(v, v)-l(u, v) l(v, u) \\
\theta((a \wedge b) \leftrightarrow(c \wedge d)) &:=l(a, c) l(b, d)-l(a, d) l(b, c)-l(d, a) l(c, b)+l(c, a) l(d, b).
\end{align*}

\noindent He then defines a map $\bar{d}:\left(\Lambda^{2} H \otimes \Lambda^{2} H\right)^{\mathfrak{S}_{2}} \rightarrow \mathbb{Z}$ by:
\begin{align*}
\bar{d}((u \wedge v) \otimes(u \wedge v))&:=0 \\
\bar{d}((a \wedge b) \leftrightarrow(c \wedge d))&:=\omega(a, b) \omega(c, d)-\omega(a, c) \omega(b, d)+\omega(a, d) \omega(b, c),
\end{align*} 

\noindent so that $\overline{q_{j}}:=\varepsilon_{j} \circ \theta+\frac{1}{3} \bar{d} \text { vanishes on } \Lambda^{4} H \subset\left(\Lambda^{2} H \otimes \Lambda^{2} H\right)^{\mathfrak{S}_2}$. Hence, it is defined on $D_2(H)$ (see diagram \eqref{diagmas}). Finally, $q_j := -\overline{q_{j}} \circ \tau_2 : \mathcal{K} \rightarrow \mathbb{Q}$ is such that \begin{equation} -\lambda_{j}=\frac{1}{24} d+q_{j}: \mathcal{K} \rightarrow \mathbb{Z}.\label{eqmorita}\end{equation}

Here comes the key point of the definition of the map $\mu$:

\begin{lemma}
For any Heegaard embedding $j$, there is a well defined map $\mu_j : D_2(H) \times \mathcal{M} \rightarrow \mathbb{Z}$ given by \[ \mu_j([T],\varphi) := ( \varepsilon_j - \varepsilon_{j \circ \varphi} )\circ \theta (T) \] for $\varphi \in \mathcal{M}$ and $T \in (\Lambda^2 H \otimes \Lambda^2 H)^{\mathfrak{S}_2}$ (here $[T]$ denotes the class of $T$ in $D_2(H)$). This map is linear in its left argument  and it satisfies: \begin{equation} (\lambda_j - \lambda_{j \circ \varphi})(h) = \mu_j(\tau_2(h), \varphi) \label{eqmu}\end{equation} for all $\varphi \in \mathcal{M}$ and $h \in \mathcal{K}$.
\label{lemma410}
\end{lemma}

\begin{proof}
For any $\varphi \in \mathcal{M}$ we have, applying \eqref{eqmorita} both to $j$ and $j \circ \varphi$, that $- ( \lambda_j - \lambda_{j \circ \varphi}) = q_j - q_{j \circ \varphi}$. This last part depends only on the second Johnson homomorphism. More precisely, by looking at the definition of $\overline{q_j}$ and $\overline{q}_{j \circ \varphi}$, one can compute that for any element $T$ in $\left(\Lambda^{2} H \otimes \Lambda^{2} H\right)^{\mathfrak{S}_2}$, whose class in $D_2(H)$ is $[T]$: 
\begin{align*}
(\overline{q_j} - \overline{q}_{j \circ \varphi})([T]) &= (\varepsilon_j - \varepsilon_{j \circ \varphi})\circ \theta(T).
\end{align*}
The result is then straightforward.
\end{proof}

\begin{rem}
Lemma \ref{lemma410} shows as explained by Morita in \cite[Rem. 6.3]{mor}, that the homomorphism $\tau_2$ contains all the information about the differences $(\lambda_j - \lambda_{j \circ \varphi})$ with $\varphi \in \mathcal{M}$. Furthermore, when reducing equation \eqref{eqmu} mod 2, one can deduce that $\beta(J_3) \subset B_0$ as claimed by Johnson in \cite[p.178]{johsurvey}. Indeed for any $f\in J_3$, and for any $\varphi \in \mathcal{M}$, we have that $\beta(f)(\omega_j) - \beta(f)(\omega_{j\circ \varphi}) = \mu_j(\tau_2(f), \varphi) = 0$, where $\omega_j$ and $\omega_{j\circ \varphi}$ are the $\operatorname{Sp}$-quadratic forms defined by the Heegaard embeddings $j$ and $j \circ \varphi$ respectively (see \cite{johquad} for more details). Hence, $\beta(f)$ is fixed by the action of $\operatorname{Sp}(2g,\mathbb{Z})$. Furthermore, it is not hard to prove from \cite{johquad} and Lemma \ref{morformula} that there exists a map $d^2$, with kernel $B_{\leq 1}$ (giving the second formal differential of boolean quadratic functions), and a commutative diagram \[\begin{tikzcd} \mathcal{K} \arrow[r, "\beta"] \arrow[d, "\tau_2"] & B_{\leq 2} \arrow[d, "d^2"] \\ D_2(H) \arrow[r] & \Lambda^2(H \otimes \mathbb{Z}_2). \end{tikzcd}\]

\noindent This implies that $\beta(J_3) \subset B_{\leq 1}$ which in turn implies that $\beta(f)$ is a constant. Indeed, there is no non-trivial $\operatorname{Sp}(2g,\mathbb{Z})$-invariant boolean affine function on the set of $\operatorname{Sp}$-quadratic forms.
\label{rembeta}
\end{rem}
\subsubsection{The application $\mu$}

We now suppose that $j$ extends to the handlebody $V$, in such a way that $j(V) = \overline{V}$ is the inner handlebody of the Heegaard splitting of $S^3$. Once such a $j$ is fixed we simply define $\mu := \mu_j$, where $\mu_j$ is defined in Lemma \ref{lemma410}. We need first the following lemma:

\begin{lemma}
For any element $T \in \tau_2(\AJ)$, the map $\mu(T,-)$  vanishes on $\mathcal{A}$.
\label{lemma56}
\end{lemma}

\begin{proof}
If we choose $\varphi$ to be in $\mathcal{A}$ and $\psi$ to be in $\AJ$, we have that $(\lambda_j - \lambda_{j \circ \varphi})(\psi) = \lambda(S^3) - \lambda(S^3) = 0$. Indeed, both $j \circ \psi \circ j^{-1}$ and $(j \circ \varphi) \circ \psi \circ (j\circ \varphi)^{-1}$ extend to the handlebody $\overline{V}$. Hence $\mu(\tau_2(\psi), \varphi)= 0$, by equation \eqref{eqmu}.
\end{proof}

\noindent Remark that whenever $\varphi$ is not in $\mathcal{A}$, then $j\circ \varphi$ does not extend to an embedding on $V$, and the conclusions of Lemma \ref{lemma56} may not be true. Also the fact that $j$ extends to $V$ is needed.

We now compute the map $\mu$ explicitly. Let $\varphi \in \mathcal{M}$ be such that $\varphi_*(A) \subset A$. Notice first that \begin{equation}\varepsilon_{j\circ\varphi}(l(u,v))  = \operatorname{lk}((j \circ\varphi)_*(u),(j \circ\varphi)^+_*(v))= \varepsilon_j(l(\varphi_*(u),\varphi_*(v)) \label{eqstarone}\end{equation} for any $u,v \in H$. We use our chosen basis for $H$ (the one defined by $j$), and write the action of $\varphi$ as a matrix $\begin{pmatrix}A&B\\0&D
\end{pmatrix}$. Then the matrix of the bilinear map $\operatorname{lk}( (j \circ \varphi)_*(-), (j \circ \varphi)_*^+(-))$ is given by
: \begin{equation} \begin{pmatrix}A&B\\0&D
\end{pmatrix}^T\begin{pmatrix}0&0\\Id&0
\end{pmatrix}\begin{pmatrix}A&B\\0&D
\end{pmatrix} =  \begin{pmatrix}0&0\\Id&D^TB
\end{pmatrix} \label{eqstartwo}\end{equation} where $S := D^TB$ is a symmetric matrix. 
We now suppose that $\varphi \in \mathcal{I}^L$, and denote $\omega_\delta$ and $\omega_S$ the pairings $H \times H \rightarrow \mathbb{Z}$ corresponding to the matrices $\begin{pmatrix}0&0\\Id&0
\end{pmatrix} $ and $\begin{pmatrix}0&0\\0&S
\end{pmatrix} $ through our choice of basis for $H$, where $S$ is the matrix describing $\sigma(\varphi)\in (A \otimes A)^{\mathfrak{S}_2}$ in the basis $(a_1, \dots, a_g)$. Note that these definitions depend on the choice of Heegaard embedding $j$.

We then have the following: 

\begin{lemma}For any $a,b,c,d,u,v$ in $H$ and for any $\varphi \in \mathcal{I}^L$, we have
\begin{align*}
-\mu\Big(\ltree{a}{b}{c}{d}, \varphi\Big) \quad &= \omega_S(a,c)\omega_S(b,d) + \omega_S(c,a)\omega_S(d,b) \\&-\omega_S(a,d)\omega_S(b,c) -\omega_S(d,a)\omega_S(c,b) \\ &+
\omega_S(a,c)\omega_\delta(b,d) + \omega_S(c,a)\omega_\delta(d,b)\\ & -\omega_S(a,d)\omega_\delta(b,c) -\omega_S(d,a)\omega_\delta(c,b) \\&+
\omega_\delta(a,c)\omega_S(b,d) + \omega_\delta(c,a)\omega_S(d,b) \\& -\omega_\delta(a,d)\omega_S(b,c) -\omega_\delta(d,a)\omega_S(c,b) \\
-\mu\Big(\frac{1}{2} \ltree{u}{v}{u}{v}, \varphi\Big)\quad &= \omega_S(u,u)\omega_S(v,v) - \omega_S(u,v)\omega_S(v,u) \\ & + \omega_S(u,u)\omega_\delta(v,v) - \omega_S(u,v)\omega_\delta(v,u) \\ & + \omega_\delta(u,u)\omega_S(v,v) - \omega_\delta(u,v)\omega_S(v,u)
\end{align*}
\noindent where $S$ is the matrix describing $\sigma(\varphi)$ in the basis $(a_1, \dots, a_g)$.
\label{computationmu}
\end{lemma}
\begin{proof}
The result follows from the definition of $\mu := \mu_j$, from the definition of $\theta$ and from: \begin{align*}
(\varepsilon_{j \circ \varphi} - \varepsilon_j)(l(a,c)l(b,d))&= \varepsilon_{j \circ \varphi}(l(a,c)l(b,d)) -\varepsilon_j(l(a,c)l(b,d))\\
&= \varepsilon_{j \circ \varphi}(l(a,c)) ~ \varepsilon_{j \circ \varphi}l(b,d)) -\varepsilon_j(l(a,c))~\varepsilon_j(l(b,d)) \\
&= (\omega_S + \omega_\delta)(a,c)(\omega_S + \omega_\delta)(b,d) -  \omega_\delta(a,c)\omega_\delta(b,d) \\
&= \omega_S(a,c)\omega_S(b,d) + \omega_S(a,c)\omega_\delta(b,d) + \omega_\delta(a,c)\omega_S(b,d)
\end{align*}

\noindent where the third equality is obtained by \eqref{eqstarone} and \eqref{eqstartwo}.
\end{proof}

We can express this is in a very compact way. Once again we define a trace-like operator $ \operatorname{Tr}^{\omega_S}$:\[
\begin{tikzcd}
 \operatorname{Tr}^{\omega_S}: D_2(H) \arrow[r] & H \otimes \mathcal{L}_{3}(H) \arrow[r, "i"]  & T_{4}(H) \arrow[r, "(\omega_S)^{1,2}"]  &  T_{2}(H)
 \end{tikzcd}\]
 
\noindent where $(\omega_S)^{1,2}$ is the contraction of the first two tensors by $\omega_S$. We now need the following lemma. 

\begin{lemma} For any $a,b,c,d,u,v \in H$, and for any $\varphi \in \mathcal{I}^L$, we have
\begin{align*}
 \operatorname{Tr}^{\omega_S}\Big(\ltree{a}{b}{c}{d}\Big) \quad &= \omega_S(a,d)(b\otimes c + c \otimes b) + \omega_S(b,c)(a\otimes d + d \otimes a) \\ &- 
\omega_S(a,c)(b\otimes d + d \otimes b) - \omega_S(b,d)(a\otimes c + c \otimes a)\\ 
 \operatorname{Tr}^{\omega_S}\Big(\frac{1}{2} \ltree{u}{v}{u}{v}\Big) \quad &= \omega_S(u,v)(u\otimes v + v \otimes u ) -\omega_S(u,u)v \otimes v - \omega_S(v,v)u \otimes u
\end{align*}

\noindent where $S$ is the matrix describing $\sigma(\varphi)$ in the basis $(a_1, \dots, a_g)$.
\label{calcultromegas}
\end{lemma}

\begin{coroll}For any $\varphi \in \mathcal{I}^L$, we have
\[ (\frac{1}{2}\omega_S + \omega_\delta) \circ  \operatorname{Tr}^{\omega_S} = \mu(-,\varphi)\]
\noindent where $S$ is the matrix describing $\sigma(\varphi)$ in the basis $(a_1, \dots, a_g)$.
\label{cor513}
\end{coroll}

\begin{proof}[Proof of Corollary \ref{cor513}]
This is a direct computation, together with the fact that the matrix~$S$ is symmetric. Set $y := (\frac{1}{2}\omega_S + \omega_\delta) \circ Tr^{\omega_S}\Big(\ltree{$a$}{$b$}{$c$}{$d$}\Big)$, then:\begin{align*}
y &= (\frac{1}{2}\omega_S + \omega_\delta)(\omega_S(a,d)(b\otimes c + c \otimes b) + \omega_S(b,c)(a\otimes d + d \otimes a) \\ & \quad - 
\omega_S(a,c)(b\otimes d + d \otimes b) - \omega_S(b,d)(a\otimes c + c \otimes a))\\
 &= \omega_S(a,d)\omega_S(b,c) + \omega_S(b,c)\omega_S(a,d) \\ &\quad-\omega_S(a,c)\omega_S(b,d) - \omega_S(b,d)\omega_S(a,c) \\
&\quad+ \omega_\delta(\omega_S(a,d)(b\otimes c + c \otimes b) + \omega_S(b,c)(a\otimes d + d \otimes a) \\ &\quad- 
\omega_S(a,c)(b\otimes d + d \otimes b) - \omega_S(b,d)(a\otimes c + c \otimes a))\\
&= \mu\Big(\ltree{$a$}{$b$}{$c$}{$d$}, \varphi\Big)
\end{align*}
where the last equality comes from Lemma \ref{computationmu}. The equality for halfs of symmetric trees can be checked in a similar way.
\end{proof}

\begin{rem}
It is easy to see that the map $\mu$ is not linear in the second variable. However, since $\omega_S \circ  \operatorname{Tr}^{\omega_S}$ clearly vanishes on $\mathcal{F}_0$, we have that the restriction $\mu_{\mid \mathcal{F}_0 \times \mathcal{I}^L}$ is bilinear, as stated in Theorem \ref{theoremtrA}.
\label{remlinear}
\end{rem}
\noindent We now prove Theorem \ref{theoremtrA}.
\begin{proof}[Proof of Theorem \ref{theoremtrA}]
Recall that $ \operatorname{Tr}^A$ vanishes on $\mathcal{F}_1$. So does $\mu$. Indeed, by Corollary \ref{cor513} and Remark \ref{remlinear}, we have $\omega_\delta \circ  \operatorname{Tr}^{\omega_S} =  \mu(-,\varphi)$ for any $\varphi \in \mathcal{I}^L$, with $S = \sigma(\varphi)$. Also, by Lemma \ref{calcultromegas}, for any $x_1,x_2 \in A$ and for any $c,d \in H$: 
\begin{align*}
\omega_\delta\Big( \operatorname{Tr}^{\omega_S}\Big(\ltree{$x_1$}{$x_2$}{$c$}{$d$}\Big)\Big) &= \omega_\delta(0)=0,\\ 
\omega_\delta\Big( \operatorname{Tr}^{\omega_S}\Big(\ltree{$x_1$}{$c$}{$d$}{$x_2$}\Big)\Big) & = \omega_\delta(\omega_S(c,d)(x_1 \otimes x_2 + x_2 \otimes x_1))= 0.
\end{align*}
For any symmetric tree $T'$ in $\mathcal{F}_1$, $\operatorname{Tr}^{\omega_S}(\frac{1}{2}T') = \frac{1}{2}\operatorname{Tr}^{\omega_S}(T')=0$. 

Hence, it is sufficient to compute the maps on trees with only one leaf colored by an element of $A$. Any half of a symmetric tree in $\mathcal{F}_0$ is actually in $\mathcal{F}_1$, and for any $a\in A$ and $c_1, c_2, c_3 \in H$, we have, once again applying Lemma \ref{calcultromegas}:

\begin{align*}
\omega_\delta\Big( \operatorname{Tr}^{\omega_S}\Big(\ltree{$a$}{$c_1$}{$c_2$}{$c_3$}\Big)\Big) \quad & = \omega_S(c_1,c_2)\omega_\delta(a \otimes c_3 + c_3 \otimes a) \\ & \quad- \omega_S(c_1,c_3)\omega_\delta(a \otimes c_2 + c_2 \otimes a) \\ 
&= \omega_S(c_1,c_2)\omega_\delta(c_3,a) - \omega_S(c_1,c_3)\omega_\delta(c_2 , a) \\ &= \omega_S(c_1,c_2)\omega'(a,{c_3}') - \omega_S(c_1,c_3)\omega'(a , {c_2}') \\&= \omega_S(c_1,c_2)\omega(a,c_3) - \omega_S(c_1,c_3)\omega(a , c_2),
\end{align*}
\noindent and, if $s :=\sum_{i,j=1}^g S_{i,j}(a_i \otimes a_j)$ is the element of $(A \otimes A)^{\mathfrak{S}_2}$ corresponding to $S$ under the isomorphism $(A \otimes A)\simeq \operatorname{Hom}(H',A)$ given by $\omega'$:
\begin{align*}
 \widetilde{\operatorname{Tr}}^A\Big(\ltree{$a$}{$c_1$}{$c_2$}{$c_3$},S\Big) \quad & = \frac{1}{2}r(s)(\omega(a,c_3)(c_2 c_1) - \omega(a,c_2)(c_3c_1)) \\ &= \omega_S(c_1,c_2)\omega(a,c_3) - \omega_S(c_1,c_3)\omega(a , c_2)
\end{align*}

\noindent as one can see by using equations \eqref{eqexemple} and \eqref{req}. Indeed, $s$ yields after dualization an element $\sum_{i,j=1}^g S_{i,j} (b_i'^* \otimes b_j'^*) \in (H'^*\otimes H'^*)$. This corresponds exactly to the element of $(H'\otimes H')^*$ induced by $\omega_S$. In other words, $r(s)(c_2c_1) = \omega_S(c_2 \otimes c_1 + c_2 \otimes c_1)=2\omega_S(c_1,c_2)$.

From these equalities, and Corollary \ref{cor513}, we can conclude that for all $T \in \mathcal{F}_0 $, and $\varphi \in \mathcal{I}^L$, $ \widetilde{\operatorname{Tr}}^A(T, \sigma(\varphi)) = \mu(T, \varphi )$. To conclude, if a tree $T$ is in $\tau_2(\AJ)$, for any $\varphi \in \mathcal{I}^L \cap \mathcal{A}$, $\mu(T,\varphi)=0$ by Lemma \ref{lemma56}. By Remark \ref{remI}, it is the same as saying that  $\mu(T,\varphi)=0$ for any $\varphi \in \mathcal{I}^L$. Remark \ref{newrem} then implies that $ \operatorname{Tr}^A(T)=0$. The map $ \operatorname{Tr}^A$ then vanishes on $\tau_2(\mathcal{A} \cap J_2)$.
\end{proof}
\begin{rem}
Note that, while the map $\mu = \mu_j : D_2(H) \times \mathcal{M} \rightarrow \mathbb{Z}$ depends on the choice of the Heegaard embedding $j: \Sigma \rightarrow S^3$ (extending to $V$), its restriction to $\mathcal{F}_0 \times \mathcal{I}^L$ only depends on the Lagrangian $A\subset H$, as a consequence of Theorem \ref{theoremtrA}.  
\end{rem}
\section{Computing $\tau_2(\mathcal{A} \cap J_2)$}
\label{sec5}
In this section we compute explicitly the image of $\AJ$ under $\tau_2$. We are going to show that it is detected by $ \operatorname{Tr}^A : \mathcal{F}_0 \rightarrow S^2(H')$. The hypothesis on the genus in the next result could probably be improved, but it would add a lot of special cases to the computations below.

\begin{theorem}
For $g \geq 4$, we have $\tau_2(\mathcal{A} \cap J_2) =   \operatorname{Ker}( \operatorname{Tr}^A) \cap \operatorname{Ker}( \operatorname{Tr}^{as}) =  \operatorname{Ker}( \operatorname{Tr}^A) \cap \operatorname{Im}(\tau_2)$.
\label{lastthm}
\end{theorem}

The inclusion $\tau_2(\mathcal{A} \cap J_2) \subset  \operatorname{Ker}( \operatorname{Tr}^A) \cap \operatorname{Ker}( \operatorname{Tr}^{as})$ follows from Theorems \ref{theoremtraces} and \ref{theoremtrA}. Recall that the elements in $D_2(H)$ are expansions of trees and halfs of symmetric trees, as explained in Section \ref{sec2}. As before, identify a tree with 4 leaves with its expansion in $D_2(H)$. A symplectic basis $(a_i,b_i)$ of $H$ is chosen so that the $a_i$'s generate the Lagrangian subgroup $A \subset H$ which is involved in the definition of $\operatorname{Tr}^A$. We denote by $B$ the Lagrangian generated by the $b_i$'s. Now, notice that trees with $0 \leq k \leq 4$ leaves colored by elements among the $a_i$'s and $4-k$ colored by elements among the $b_i$'s give, after projection, generators of the quotient $\mathcal{F}_{k-1} / \mathcal{F}_k$. We call such trees \emph{trees of type $k$}. For example $\mathcal{F}_0/\mathcal{F}_1$ is generated by trees of type $1$. Also an element of $\mathcal{F}_0$ can be written as a linear combination of elements of type $1$ to $4$.  

We will use several times the following lemma.

\begin{lemma}
Let \ses{K}{F}{C} be a short exact sequence of finitely generated abelian groups. We suppose that $C$ is a free abelian group or a $\mathbb{Z}_2$-vector space, and that we have a generating family $(f_i)_{0 \leq i \leq f}$ of $F$ and a basis $(c_j)_{0 \leq j \leq c}$ for $C$ such that $(f_i)$ consists of elements of $K$ and lifts of elements of $(c_j)$. Then, $K$ is generated by \[(\{f_i \mid 0 \leq i \leq f \} \cap K) \cup (\{f_i - f_j \mid 0 \leq i < j \leq f \} \cap K)\] if $C$ is free abelian and by \[(\{f_i \mid 0 \leq i \leq f \} \cap K) \cup (\{f_i - f_j \mid 0 \leq i < j \leq f \} \cap K) \cup \{2f_i \mid 0 \leq i \leq f \}\] if $C$ is a $\mathbb{Z}_2$-vector space.
\label{linalglemma}
\end{lemma}

\begin{proof}
Let us suppose that $C$ is a free abelian group. Let $p : F \rightarrow C$ denote the projection. By our hypothesis, for every $j \in \lbrace 0, \dots , c \rbrace$, $c_j$ has a lift among the $f_i$'s. We denote $(k_i)_{0 \leq i \leq \kappa}$ the elements among the $f_i$'s that are in $K$ and $(l_i)_{0 \leq i \leq l}$ the other ones. Now, for any $x \in K$, we can write $x = \sum \limits_{i \leq \kappa} \lambda_i k_i + \sum \limits_{j \leq l} \mu_j \l_j$, with $\lambda_i, \mu_j \in \mathbb{Z}$. We thus have $( \sum  \limits_{p(l_j) = c_i} \mu_j )c_i = 0 $, hence  $\sum \limits_{p(l_j) = c_i}\mu_j = 0$ for every $i \in \lbrace 0, \dots , c \rbrace$. Fix $1 \leq i \leq c$ and consider the $l_j$'s such that $p(l_j) = c_i$, and renumber in a simpler way the elements (denoted by $\mu'$ and $l'$ after renumbering) from $0$ to $n_i$ such that: $\sum \limits_{p(l_j) = c_i} \mu_j l_j = \sum \limits_{j=0}^{n_i} \mu'_j l'_j = \sum \limits_{j=0}^{n_i} \sum \limits_{s=0}^{j-1} \mu_j' (l'_{s+1}-l'_{s})$, where we used that $\sum \limits_{j=0}^{n_i} \mu'_j = 0$. This computation allows us to write $x$ as a linear combination of the $k_i$'s and elements $l_i - l_j$ such that ${p(l_i) = p(l_j)}$. For the case where $C$ is a $\mathbb{Z}_2$-vector space, the proof can be easily adapted.
\end{proof}

\begin{rem}
The generating family provided by Lemma \ref{linalglemma} is far from being optimal. For example, given $x,y,z \in F$ with the same image in $C$, one does not need to take $(x-y)$, $(x-z)$ and $(y-z)$, as the last one is a linear combination of the other two.
\end{rem}

Let $T$ be in $\mathcal{F}_0 \cap  \operatorname{Im}(\tau_2)$ and write it as $T_1 + T_{\geq 2}$  where $T_1$ and $T_{\geq 2}$ are written as some linear combinations of respectively type $1$ elements and type $2$ to $4$ elements. We suppose that $ \operatorname{Tr}^A(T_1 + T_{\geq 2}) = 0$ i.e. $ \operatorname{Tr}^A(T_1) = 0$. Using the special elements of $\AJ$ described in Section \ref{section51} we are going to show that $T \in \tau_2(\mathcal{A} \cap J_2)$. 

From now on, we refer to the element in $\tau_2(\mathcal{A} \cap J_2)$ as \emph{realizable} elements. We also say that a tree of type $0$ to $4$ has a \emph{contraction} when at least two of its leaves can be paired non-trivially through $\omega$. Some of the computations below are inspired by computations in \cite{mor} and \cite{pit}.

For the sake of preciseness, we emphasize that for two submodules $P$ and $Q$ of a module $V$, the notation $P \wedge Q$ stands for the image of $P \otimes Q$ under the projection $V \otimes V \rightarrow \Lambda^2(V)$. From Section \ref{sec2} we have that $ \operatorname{Tr}^{as}$ vanishes on elements of $ \operatorname{Im}(\tau_2)$. On an element of type~$1$ this trace takes value in $(B \wedge B) \otimes \mathbb{Z}_2$ and on other types it takes values in $(A \wedge H)\otimes \mathbb{Z}_2$. Hence, using the decomposition $ \Lambda^2 H = (B \wedge B) \oplus (A \wedge H)$, it is clear that $ \operatorname{Tr}^{as}(T_1 + T_{\geq 2}) = 0$ implies  $\operatorname{Tr}^{as}(T_1) =   \operatorname{Tr}^{as}(T_{\geq 2}) = 0$. In the sequel, we shall prove that $T_1 \in \tau_2(\AJ)$ and, next, we will show that  $T_{\geq 2} \in \tau_2(\AJ)$ using the fact that $\operatorname{Tr}^{as}(T_{\geq 2})=0$.

In terms of the symplectic basis $(a_i,b_i)$ of $H$, the elements of type $1$ can be of the following form (up to sign): \[ {\circled{1}}_{i,j,k,l} := \ltree{$a_i$}{$b_j$}{$b_k$}{$b_l$} \quad \quad \quad \quad {\circled{2}}_{i,j,k} :=\ltree{$a_i$}{$b_j$}{$b_k$}{$b_i$}\] \[ {\circled{3}}_{i,k,l} :=\ltree{$a_i$}{$b_i$}{$b_k$}{$b_l$} \quad \quad \quad \quad {\circled{4}}_{i,k} :=\ltree{$a_i$}{$b_i$}{$b_k$}{$b_i$}\]

\noindent with $i$ different from $j,k$ and $l$. 

\begin{propo}
Set $N :=  \operatorname{Ker}( \operatorname{Tr}^A : \operatorname{Span}_\mathbb{Z}\{ \textit{type 1 elements} \} \rightarrow S^2(H'))$. Then $N$ is generated by elements of type $\circled{$1$},\circled{$3$}$ and \[ \circled{$2$}_{i,j,j} - \circled{$2$}_{i',j,j}~; \quad \circled{$2$}_{i,j,k} - \circled{$2$}_{i',j,k}~; \quad \circled{$2$}_{i,j,k} - \circled{$2$}_{i',k,j}~; \]
\[ \circled{$2$}_{i,j,k} - \circled{$4$}_{j,k}~; \quad \circled{$2$}_{i,j,k} - \circled{$4$}_{k,j}~; \quad \circled{$4$}_{i,k} - \circled{$4$}_{k,i}~;\]
\noindent where $i$ and $i'$ must be different from $j,k$ and $l$, and $j \neq k$.
\label{defN}
\end{propo}

\begin{proof}
It is a consequence of Lemma \ref{linalglemma} applied to the short exact sequence \ses{N}{\operatorname{Span}_\mathbb{Z}\{ \textit{type 1 elements} \}}{S^2(H')} after computing that \begin{align*}
 \operatorname{Tr}^{A}(\circled{1}_{i,j,k,l}) &= 0 \\
 \operatorname{Tr}^{A}(\circled{3}_{i,k,l}) &=  0\\
 \operatorname{Tr}^{A}(\circled{2}_{i,j,k}) &= +b_k'b_j' \\
 \operatorname{Tr}^{A}(\circled{4}_{i,k}) &= + b_k'b_i'.
\end{align*}
Here, the generating family for $\operatorname{Span}_\mathbb{Z}\{ \textit{type 1 elements} \}$ is the family of type 1 elements, and the basis we use for $S^2(H')$ is $(b_i'b_j')_{1 \leq i \leq j \leq g}$. 
\end{proof}
We are going to show that $N \subset \tau_2(\AJ)$, in particular we will have $T_1 \in \tau_2(\AJ)$. First, ${\circled{1}}_{i,j,k,l}$ can be written as $\ltree{$a_i$}{$b_j$}{$b_k$}{$b_l$} = \Bigg[ -\ltritree{$a_i$}{$b_j$}{$b_m$} ,\ltritree{$b_l$}{$a_m$}{$b_k$} \Bigg]$ where $m$ is different from $j$. Morita has shown, as stated in Proposition \ref{morprop}, that (the expansion of) a tree (with three leaves) in $D_1(H)$ is in $\tau_1(\mathcal{A}\cap J_1)$ if and only if one of the leaves vanishes in $H'$ \cite{mor93}, where $D_1(H)= \Lambda^3 H$ has been identified with $\mathcal{A}_1^t(H)$. Hence ${\circled{1}}_{i,j,k,l}$ is indeed in $\tau_2(\mathcal{A} \cap J_2)$, obtained as the image by $\tau_2$ of an element of the third family defined in Section \ref{section51}: a commutator of the Torelli handlebody group. This is also true for ${\circled{3}}_{i,k,l} = \Bigg[-\ltritree{$a_i$}{$b_i$}{$b_m$} ,\ltritree{$b_l$}{$a_m$}{$b_k$} \Bigg]$ with $m \neq i$. Now, we are left with the generators of $N$ built in Proposition \ref{defN} from $\circled{2}$ and $\circled{4}$ elements. One can check that:
\begin{align*}
\circled{2}_{i,j,j} - \circled{2}_{i',j,j} & = \Bigg[-\ltritree{$a_i$}{$b_j$}{$b_{i'}$} ,\ltritree{$b_i$}{$a_{i'}$}{$b_j$} \Bigg] \\
\circled{2}_{i,j,k} - \circled{2}_{i',k,j} & = \Bigg[-\ltritree{$a_i$}{$b_j$}{$b_{i'}$} ,\ltritree{$b_i$}{$a_{i'}$}{$b_k$} \Bigg] \\ \circled{2}_{i,j,k} - \circled{2}_{i',j,k} &=  \Bigg[-\ltritree{$a_i$}{$b_j$}{$b_{i'}$} ,\ltritree{$b_i$}{$a_{i'}$}{$b_k$} \Bigg] {+} \ltree{$a_{i'}$}{$b_{i'}$}{$b_j$}{$b_k$},
\end{align*}
\noindent and that:
\begin{align*}
\circled{2}_{i,j,k} - \circled{4}_{j,k}&= \Bigg[-\ltritree{$a_i$}{$b_j$}{$b_{l}$} ,\ltritree{$b_i$}{$a_{l}$}{$b_k$} \Bigg] + \Bigg[\ltritree{$a_j$}{$b_j$}{$b_{l}$} ,\ltritree{$b_j$}{$a_{l}$}{$b_k$} \Bigg] \\
\circled{2}_{i,j,k} - \circled{4}_{k,j} &= \quad (\circled{2}_{i,j,k} - \circled{4}_{j,k}) + (\circled{4}_{j,k} - \circled{4}_{k,j})
\end{align*}
\begin{align*}
\circled{4}_{i,k} - \circled{4}_{k,i} &= \Bigg[-\ltritree{$a_i$}{$b_i$}{$b_{l}$} ,\ltritree{$b_i$}{$a_{l}$}{$b_k$} \Bigg] + \Bigg[\ltritree{$a_k$}{$b_k$}{$b_{l}$} ,\ltritree{$b_k$}{$a_{l}$}{$b_i$} \Bigg]\\ &{+} \ltree{$a_l$}{$b_{l}$}{$b_i$}{$b_k$},
\end{align*}

\noindent with $l$ always chosen so that it does not add any contraction, which is possible if the genus is greater or equal to $4$. We know how to show that each of the terms are in $\tau_2(\AJ)$, because:

\begin{align*}
\ltree{$a_{i'}$}{$b_{i'}$}{$b_j$}{$b_k$} &=  \Bigg[ - \ltritree{$a_{i'}$}{$b_{i'}$}{$b_{l}$} ,\ltritree{$b_k$}{$a_{l}$}{$b_j$} \Bigg] \textit{ with } l \neq i' \\ 
\end{align*}

\noindent for $i' \neq j,k$; and all of these terms are in the image of elements of the third kind described in Section \ref{section51}. Hence $N \subset \tau_2(\AJ)$.
\begin{rem}
One can notice that all the trees that have been used above to realize elements of $N$ as linear combination of Lie brackets of elements of $\tau_1(\mathcal{A}\cap J_1)$ are colored by elements of $A$ \emph{and} elements of $B$ (and never only by $A$ or only by $B$).
\label{remcoloration}
\end{rem}
We now turn to the element $T_{\geq 2}$. We remark that $(A \wedge H) = (A\wedge A)  \oplus (A\wedge B)$, hence if write $T_{\geq 2} = T_2 +T_{\geq 3}$ where $T_2$ is a linear combination of type two elements and $T_{\geq 3}$ a linear combination of type 3 and 4 elements, then we have $ \operatorname{Tr}^{as}(T_2) =  \operatorname{Tr}^{as}(T_{\geq 3}) =0$, because $\operatorname{Tr}^{as}(T_{\geq 2}) = 0$. We will deal first with $T_2$. By the $IHX$ relation, we can even restrict our type two elements appearing in the writing of $T_2$ to trees where the two $A$ colors are not ``close" to each other, i.e. trees of the form: \[ {\circled{5}}_{i,j,k,l} := \ltree{$a_i$}{$b_j$}{$b_k$}{$a_l$} \quad {\circled{6}}_{i,j} := \frac{1}{2} \ltree{$a_i$}{$b_j$}{$a_i$}{$b_j$} \]

\noindent with no conditions on the indices. It is known, by Morita's formula in Lemma \ref{morformula}, that the elements of the form ${\circled{6}}_{i,i}$ can be obtained as the image under $\tau_2$ of a Dehn twist along a curve $\gamma_{i}$ bounding a subsurface with $a_i, b_i$ forming a symplectic basis of this subsurface. This curve can be chosen to bound a disk in the handlebody (see Figure \ref{surfcurves}) so that the corresponding Dehn twist is an element of the first kind described in Section \ref{section51}. Hence, ${\circled{6}}_{i,i}$ belongs to $\tau_2(\AJ)$, and we now suppose $i\neq j$ in the definition of ${\circled{6}}_{i,j}$. 

We show that ${\circled{5}}_{i,i,j,j}$ is realizable. This will be useful in the computations below. Take disjoints neighborhoods of, respectively, $\alpha_i \cup \beta_i$ and $\alpha_j \cup \beta_j$, and band this two genus 1 surfaces as shown in Figure \ref{surfcurves}. The boundary $\gamma_{i,j}$ of the resulting genus 2 surface is bounding a disk in the handlebody and its image by $\tau_2$ (using Lemma \ref{morformula}) is: \[\tau_2(T_{\gamma_{i,j}}) = {\circled{6}}_{i,i} - {\circled{5}}_{i,i,j,j} +{\circled{6}}_{j,j}\] which shows that ${\circled{5}}_{i,i,j,j} \in \tau_2(\AJ)$. 

\begin{figure}[h]
	\centering
	\includegraphics[scale= 0.2]{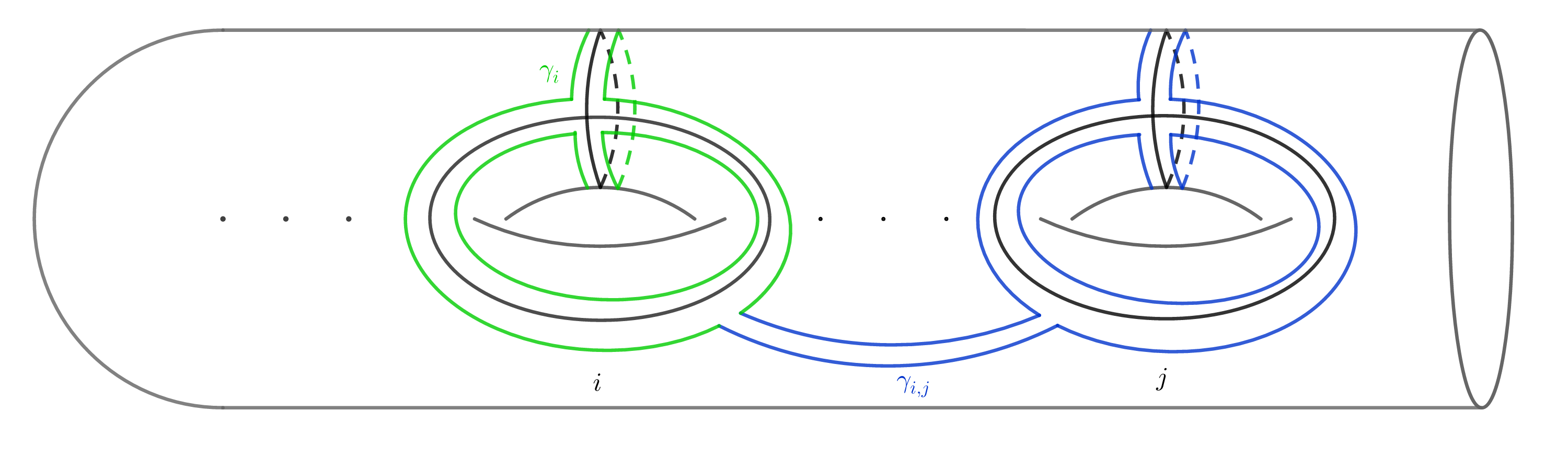}
	\caption{Curves $\gamma_i$ and $\gamma_{i,j}$.}
	\label{surfcurves}
\end{figure}

We divide cases in terms of the number of leaves that contract in $\circled{5}_{i,j,k,l}$. If there is no contraction ($j\neq l$ and $k\neq i$), then $\circled{5}_{i,j,k,l}$ can be easily obtained as a commutator of trees with a leaf in $A$, supposing $g \geq 4$. If there are 2 contractions, then $k=i$ and $j=l$, which yields two cases: if $i=j$ then we get $-2{\circled{6}}_{i,i}$, which we have already dealt with; if not, we get an element ${\circled{5}}_{i,j,i,j} \notin \operatorname{Ker}(\operatorname{Tr}^{as})$. If there is only one contraction, then up to symmetry ($\circled{5}_{i,j,k,l} = \circled{5}_{l,k,j,i}$) we can suppose that $k=i$ and $j\neq l$. Hence the remaining element $T'_{2}$ (the part of $T_2$ which is not yet proved to be in $\tau_2(\AJ)$) is a linear combination of trees of the form ${\circled{5}}_{i,j,i,j}$ (with $i \neq j$), ${\circled{5}}_{i,j,i,l}$ (with $l \neq j$) and ${\circled{6}}_{i,j}$ (with $i \neq j$) such that :

\begin{align*}
\operatorname{Tr}^{as}({\circled{5}}_{i,j,i,j}) &= a_i \wedge b_i + a_j \wedge b_j \quad \textit{with }i \neq j \\
 \operatorname{Tr}^{as}({\circled{5}}_{i,j,i,l}) &= a_l \wedge b_j \quad \textit{with } l\neq j \\
 \operatorname{Tr}^{as}({\circled{6}}_{i,j}) &= a_i \wedge b_j \quad \textit{with }i \neq j,
\end{align*}
and satisfies $\operatorname{Tr}^{as}(T_{2}') = 0$.

Notice that in $\operatorname{Ker}(\omega : \Lambda^2(H/2H) \rightarrow \mathbb{Z}_2)$ the two subspaces $\operatorname{Span}_{\mathbb{Z}_2} \{ a_i \wedge b_j \mid i \neq j \}$ and $\operatorname{Span}_{\mathbb{Z}_2} \{ a_i \wedge b_i + a_j \wedge b_j \mid i \neq j \}$ have trivial intersection. This allow us to write $T'_2$ as a sum of two elements, say $U$ and $V$, the element $U$ being in $\operatorname{Span}_\mathbb{Z} \Big \{ {\circled{5}}_{i,j,i,l},{\circled{6}}_{i,j} \mid j \neq l \Big \} $ and $V$ being in $\operatorname{Span}_{\mathbb{Z}} \Big \{ \circled{5}_{i,j,i,j} \mid i \neq j\Big \}$, such that $\operatorname{Tr}^{as}(U)= \operatorname{Tr}^{as}(V) = 0$. The element $U$ can be written as a linear combination of \[ {\circled{5}}_{i,j,i,l} \pm {\circled{5}}_{i',j,i',l} \quad; \quad 2{\circled{6}}_{i,j} \quad ; \quad {\circled{5}}_{i,j,i,l} \pm {\circled{6}}_{l,j}\]

\noindent by Lemma \ref{linalglemma} applied to the short exact sequence
\begin{equation}
\begin{tikzcd}[sep= small]
 0 \arrow[r]& K \arrow[r] & \operatorname{Span}_\mathbb{Z} \Big \{ {\circled{5}}_{i,j,i,l} ~,~{\circled{6}}_{i,j} \mid j \neq l \Big \}  \arrow[r, "\operatorname{Tr}^{as}"]& \operatorname{Span}_{\mathbb{Z}_2} \{ a_i \wedge b_j \mid i \neq j \} \arrow[r] &0 
\end{tikzcd}
\label{defK}
\end{equation}
where the generating family for $\operatorname{Span}_\mathbb{Z} \Big \{ {\circled{5}}_{i,j,i,l}~,~{\circled{6}}_{i,j} \mid j \neq l \Big \} $ and the basis for $\operatorname{Span}_{\mathbb{Z}_2}\{a_i~\wedge~b_j\mid i \neq j \}$ are given in their definition. The tree $2{\circled{6}}_{i,j}$ has no contractions and can be realized as a commutator of the Torelli handlebody group. We also have, with $r \neq i,j,l$:

\begin{align*}
\circled{5}_{i,j,i,l} + \circled{5}_{i',j,i',l} & = \Bigg[ \ltritree{$a_i$}{$b_j$}{$a_{i'}$} ,\ltritree{$a_l$}{$b_{i'}$}{$b_i$} \Bigg] \textit{  as }l\neq j,\end{align*}
\begin{align*}
2\circled{5}_{i,j,i,l} & = \Bigg[ \ltritree{$a_i$}{$b_j$}{$a_r$} ,\ltritree{$a_l$}{$b_{r}$}{$b_i$} \Bigg] - \Bigg[ \ltritree{$a_i$}{$b_j$}{$b_r$} ,\ltritree{$a_l$}{$a_{r}$}{$b_i$} \Bigg] \\&{-} \Bigg[ \ltritree{$a_r$}{$b_r$}{$a_i$} ,\ltritree{$a_l$}{$b_{i}$}{$b_j$} \Bigg]  \textit{  as }l\neq j,
\end{align*}

\noindent and with the same arguments as above these elements are realizable (using the first family described in Section \ref{section51}). For elements involving $\circled{5}$ and $\circled{6}$, if $i \neq j$, we have: 

\begin{align*}
\circled{6}_{l,j} - \circled{5}_{i,j,i,l} & = \circled{6}_{l,j} + \Bigg[- \ltritree{$a_i$}{$b_j$}{$a_{j}$} ,\ltritree{$a_l$}{$b_{j}$}{$b_i$} \Bigg]  {-} \ltree{$b_j$}{$a_j$}{$b_j$}{$a_l$}.
\end{align*}

\noindent As $\circled{6}_{l,j} - \ltree{$b_j$}{$a_j$}{$b_j$}{$a_l$} + \circled{6}_{j,j}$ can be obtained from the Dehn twist along the boundary of a neighborhood of $(\alpha_j \sharp \alpha_l^{-1}) \cup \beta_j$ (where $(\alpha_j \sharp \alpha_l^{-1})$ denotes the connected sum of $\alpha_j$ and $\alpha_l$), and knowing that $\circled{6}_{j,j}$ is in $\tau_2(\AJ)$ we conclude that $\circled{6}_{l,j} - \circled{5}_{i,j,i,l}$ is realizable for $i \neq j$. If $i = j$, then we write, for some $j' \neq i$:  \begin{align}
\circled{6}_{l,j} - \circled{5}_{j,j,j,l} = (\circled{6}_{l,j} - \circled{5}_{j',j,j',l}) + (\circled{5}_{j',j,j',l}-\circled{5}_{j,j,j,l})
\label{elem5}
\end{align}  \noindent and we have just shown that both terms of this sum are realizable. We conclude that $U$ is realizable.

We need to show that $V$ is also realizable, which will show that $T'_2$, and hence $T_2$ are also realizable. We need the following.
\begin{lemma}
The kernel $S$ in the short exact sequence \[ 
\begin{tikzcd}
 0 \arrow[r]& S \arrow[r] & \operatorname{Span}_{\mathbb{Z}} \Big \{ {\circled{$5$}}_{i,j,i,j} \mid i \neq j  \Big \}  \arrow[r, "\operatorname{Tr}^{as}"]& \operatorname{Span}_{\mathbb{Z}_2} \{ a_i \wedge b_i + a_j \wedge b_j \mid i \neq j \} \arrow[r] &0 
\end{tikzcd} 
\] \noindent is generated by the family \[ \Big \{ {\circled{5}}_{1,i,1,i} + {\circled{5}}_{i,j,i,j}  + {\circled{5}}_{j,1,j,1} \mid i \neq j \Big \} \cup  \Big \{ 2{\circled{5}}_{i,j,i,j} \mid i \neq j \Big \}. \]
\label{lemmeS}
\end{lemma}

\begin{proof}
It is not hard to see, by sending the family $ \Big \{ {\circled{$5$}}_{i,j,i,j} \mid i <j \Big \} $ to $(H\otimes \mathbb{Q})^{\otimes 4}$ through the expansion map and the inclusion $\mathcal{L}(H \otimes \mathbb{Q}) \subset T(H \otimes \mathbb{Q})$, that this family is free. Indeed, ${\circled{$5$}}_{i,j,i,j}$ is sent to a sum of 16 terms, from each of which one can recover $i$ and $j$. Each of these terms belongs (up to a sign) to the basis of $(H\otimes \mathbb{Q})^{\otimes 4}$ induced by the basis chosen for~$H$. Hence, $\operatorname{Span}_{\mathbb{Z}} \Big \{ {\circled{$5$}}_{i,j,i,j} \mid i \neq j  \Big \}$ is free and we can apply Lemma \ref{linalglemma} to the short exact sequence by using the family $ \Big \{ {\circled{$5$}}_{i,j,i,j} - {\circled{$5$}}_{j,1,j,1} \mid i <j \Big \} \cup \Big \{ {\circled{$5$}}_{j,1,j,1} \mid 2 \leq j \Big \}$ as a generating family for $\operatorname{Span}_{\mathbb{Z}}\Big \{ {\circled{$5$}}_{i,j,i,j} \mid i \neq j \} \Big \}$, and the basis $ (a_1 \wedge b_1 + a_i \wedge b_i)_{2 \leq i \leq g} $ for $\operatorname{Span}_{\mathbb{Z}_2} \{ a_i \wedge b_i + a_j \wedge b_j \mid i \neq j \}$. Then $S$ is generated by $ \Big \{ 2{\circled{$5$}}_{i,j,i,j} - 2{\circled{$5$}}_{j,1,j,1} \mid i <j \Big \} \cup \Big \{ 2{\circled{$5$}}_{j,1,j,1} \mid 2 \leq j \Big \} \cup \Big \{ {\circled{$5$}}_{i,j,i,j} - {\circled{$5$}}_{j,1,j,1} - {\circled{$5$}}_{i,1,i,1} \mid  i < j \Big \} \cup \Big \{ {\circled{$5$}}_{i,j,i,j} - {\circled{$5$}}_{j,1,j,1} - {\circled{$5$}}_{i,k,i,k}  + {\circled{$5$}}_{k,1,k,1} \mid  i < j < k\Big \}$, from which we deduce the simpler generating family \[ \Big \{ {\circled{5}}_{1,i,1,i} + {\circled{5}}_{i,j,i,j}  + {\circled{5}}_{j,1,j,1} \mid i \neq j \Big \} \cup \Big \{ 2{\circled{5}}_{i,j,i,j} \mid i \neq j \Big \}. \] Indeed, it is easy to get the elements of the family given by Lemma \ref{linalglemma} with the elements given right above. For example, for $ i < j$: \begin{align*}
{\circled{$5$}}_{i,j,i,j} - {\circled{$5$}}_{j,1,j,1} - {\circled{$5$}}_{i,1,i,1} &= ({\circled{5}}_{1,i,1,i} + {\circled{5}}_{i,j,i,j}  + {\circled{5}}_{j,1,j,1}) - 2{\circled{5}}_{1,i,1,i} - 2{\circled{5}}_{j,1,j,1},
\end{align*}

\noindent and for $ i < j <k$ :
\begin{align*}
{\circled{$5$}}_{i,j,i,j} - {\circled{$5$}}_{j,1,j,1} - {\circled{$5$}}_{i,k,i,k}  + {\circled{$5$}}_{k,1,k,1} &= ({\circled{5}}_{1,i,1,i} + {\circled{5}}_{i,j,i,j} + {\circled{5}}_{j,1,j,1} - 2{\circled{5}}_{j,1,j,1}) \\ &- ({\circled{5}}_{1,i,1,i} + {\circled{5}}_{i,k,i,k} + {\circled{5}}_{k,1,k,1} - 2{\circled{5}}_{k,1,k,1}).
\end{align*}

\end{proof}

Hence by Lemma \ref{lemmeS}, the element $V$ can be written as a linear combination of \[ {\circled{5}}_{1,i,1,i} + {\circled{5}}_{i,j,i,j} + {\circled{5}}_{j,1,j,1}~; \quad 2{\circled{5}}_{i,j,i,j}\] \noindent where $i \neq j$. We compute, for $i,j \neq 1$: \begin{align*}
\Bigg[ \ltritree{$a_1$}{$b_i$}{$a_{j}$} ,\ltritree{$a_i$}{$b_{j}$}{$b_1$} \Bigg] &=  \circled{5}_{1,i,1,i} + \circled{5}_{i,j,i,j} - \ltree{$a_j$}{$a_1$}{$b_j$}{$b_1$}\\
& = \circled{5}_{1,i,1,i} + \circled{5}_{i,j,i,j} -  \circled{5}_{j,1,j,1} +  \circled{5}_{j,j,1,1}.
\end{align*}

\noindent We know that $\circled{5}_{j,j,1,1}$ is realizable which shows that $\circled{5}_{1,i,1,i} + \circled{5}_{i,j,i,j} -  \circled{5}_{j,1,j,1}$ is also realizable. Similarly, the element $\circled{5}_{i,j,i,j} + \circled{5}_{j,1,j,1} - \circled{5}_{1,i,1,i}$ is realizable. By summing these two elements, we get that $ 2{\circled{5}}_{i,j,i,j}$ is also realizable. We deduce that both ${\circled{5}}_{1,i,1,i} + {\circled{5}}_{i,j,i,j} + {\circled{5}}_{j,1,j,1}$ and $2{\circled{5}}_{i,j,i,j}$ belong to $\tau_2(\AJ)$ for any $i \neq j$. Therefore $V$ is realizable.
We finally turn to $T_{\geq 3}$. We define the elements \[ {\circled{7}}_{i,j,k,l} := \ltree{$a_i$}{$a_j$}{$b_k$}{$a_l$} \quad {\circled{8}}_{i,j} := \frac{1}{2} \ltree{$a_i$}{$a_j$}{$a_i$}{$a_j$} \]
\noindent with $i\neq j$, then \begin{align*}
 \operatorname{Tr}^{as}({\circled{7}}_{i,j,k,l}) &= \delta_{ki}a_j \wedge a_l + \delta_{kj}a_i \wedge a_l \\
 \operatorname{Tr}^{as}({\circled{8}}_{i,j}) &= a_i \wedge a_j
\end{align*}

The fact that $T_{\geq 3}$ can be realized will follow from the same kind of computations as for~$T_2$. We define $P:=  \operatorname{Ker}( \operatorname{Tr^{as}} : \operatorname{Span}_\mathbb{Z} \{ \textit{type 3 and 4 elements} \} ~ \rightarrow (A \wedge A) \otimes \mathbb{Z}_2)$.

\begin{propo}
$P$ is generated by trees with 4 leaves colored by $A$, elements of type $\circled{$7$}$ with no contractions, elements of type ${\circled{$7$}}_{i,k,k,i}$ and elements \[ \circled{$7$}_{i,k,k,m} \pm \circled{$7$}_{m,k,k,i}~; \circled{$7$}_{i,k,k,l} \pm \circled{$7$}_{i,k',k',l}~; \circled{$7$}_{i,k,k,l} \pm \circled{$8$}_{i,l} \]
\noindent where $i$ must be different from $k,k'$ and $l$, and $m \neq k$.
\end{propo}

\begin{proof}
It follows once again from Lemma \ref{linalglemma} applied to the short exact sequence \[ 
\begin{tikzcd}
 0 \arrow[r]& P \arrow[r] & \operatorname{Span}_\mathbb{Z} \{ \textit{type 3 and 4 elements} \}  \arrow[r, "\operatorname{Tr}^{as}"]& ((A \wedge A) \otimes \mathbb{Z}_2) \arrow[r] &0 
\end{tikzcd}. 
\]
Type 3 and 4 elements give a generating family for $\operatorname{Span}_\mathbb{Z} \{ \textit{type 3 and 4 elements} \}$ and $(a_i \wedge a_j)_{1 \leq i < j \leq g}$ a basis for $((A \wedge A) \otimes \mathbb{Z}_2)$. Note that, according to Lemma $\ref{linalglemma}$, in our family of generators we should have elements of type $2 \circled{$7$}_{i,j,k,l}$ and $2\circled{$8$}_{i,j}$  for any $i \neq j$. Nevertheless, these elements are not needed, because if there is no contraction we have the element $\circled{$7$}_{i,j,k,l}$ as a generator and if there is one contraction it is easy to obtain both $2 \circled{$7$}_{i,j,j,l}$ and $2 \circled{$7$}_{i,j,i,l} = -2 \circled{$7$}_{j,i,i,l}$ from the generators given in the proposition. This last argument also works for $2\circled{$8$}_{i,j}$.
\end{proof}

We now show that $P$ is contained in $\tau_2(\AJ)$. Elements of type 4 that are not expansion of half trees in $D_2(H)$ are not worth mentioning: they always have no contractions and are in $\tau_2([\mathcal{A}\cap J_1,\mathcal{A}\cap J_1])$. The same is true for elements of type $\circled{7}$ with no contractions. Once again we check some relations, making sure that any tree with three leaves appearing in the computations below has at least one leaf colored by $A$:

\begin{align*}
\circled{7}_{i,k,k,l} + \circled{7}_{i,k',k',l}& = \Bigg[\ltritree{$a_i$}{$a_k$}{$a_{k'}$} ,\ltritree{$a_l$}{$b_{k'}$}{$b_k$} \Bigg], \\
\circled{7}_{i,k,k,l} - \circled{7}_{l,k',k',i}& = \Bigg[-\ltritree{$a_i$}{$a_k$}{$b_{k'}$} ,\ltritree{$a_l$}{$a_{k'}$}{$b_k$} \Bigg], \\ \circled{7}_{i,k,k,m} - \circled{7}_{m,k,k,i}& = \ltree{$a_k$}{$b_{k}$}{$a_m$}{$a_i$} \\ &=  \Bigg[\ltritree{$a_k$}{$b_k$}{$a_{i}$} ,\ltritree{$a_i$}{$b_{i}$}{$a_m$} \Bigg],
\end{align*}
\begin{align*}
 2\circled{7}_{i,k,k,m} & = \Bigg[\ltritree{$a_i$}{$a_k$}{$a_{r}$} ,\ltritree{$a_m$}{$b_{r}$}{$b_k$} \Bigg] - \Bigg[\ltritree{$a_i$}{$a_k$}{$b_{r}$} ,\ltritree{$a_m$}{$a_{r}$}{$b_k$} \Bigg] \\ &- \Bigg[\ltritree{$a_r$}{$b_r$}{$a_{k}$} ,\ltritree{$a_i$}{$b_{k}$}{$a_m$} \Bigg], \\ \circled{$8$}_{i,l} \pm \circled{$7$}_{i,k,k,l} &= (\circled{$8$}_{i,l} \pm \circled{$7$}_{i,l,l,l}) \pm (\circled{$7$}_{i,k,k,l} \mp \circled{$7$}_{i,l,l,l}).
\end{align*}

We also consider the Dehn twist along the curve bounding the surface which is a neighborhood of $\alpha_l \cup (\alpha_i \sharp \beta_l^{\pm})$, where $(\alpha_i \sharp \beta_l^{\pm})$ is a connected sum of $\alpha_i$ and $\beta_l$ with orientation defined by the sign. This element is in $\AJ$, its image under $\tau_2$ is \[ \frac{1}{2} \ltree{$a_i \pm b_l$}{$a_l$}{$a_i \pm b_l$ }{$a_l$} = \circled{$8$}_{i,l} \pm \circled{$7$}_{i,l,l,l} + \frac{1}{2} \ltree{$b_l$}{$a_l$}{$b_l$ }{$a_l$}\]

\noindent which ultimately shows that $\circled{$8$}_{i,l} \pm \circled{$7$}_{i,l,l,l}$ belongs to $\tau_2(\AJ)$. Therefore, $\circled{$8$}_{i,l} \pm \circled{$7$}_{i,k,k,l}$ is also realizable. Finally, elements of type $\circled{7}_{i,k,k,i}$ can be realized in the following way. Notice that: \begin{align*}
\frac{1}{2} \ltree{$a_i + a_k$}{$b_k +a_i$}{$a_i + a_k$}{$b_k +a_i$} - \frac{1}{2} \ltree{$a_k$}{$b_k +a_i$}{$a_k$}{$b_k +a_i$} &= \frac{1}{2} \ltree{$a_i$}{$b_k +a_i$}{$a_i$}{$b_k +a_i$} + \ltree{$a_i$}{$b_k+a_i$}{$a_k$}{$b_k +a_i$} \\ &= \frac{1}{2} \ltree{$a_i$}{$b_k$}{$a_i$}{$b_k$} + \ltree{$a_i$}{$b_k$}{$a_k$}{$b_k +a_i$} \\ &= \circled{$6$}_{i,k} - \circled{$5$}_{i,k,k,k} + \circled{$7$}_{i,k,k,i}.
\end{align*}
Now, we have already shown that $\circled{$6$}_{i,k} - \circled{$5$}_{i,k,k,k} \in \tau_2(\AJ)$ (by equality \eqref{elem5}), because it can be written as $\circled{$6$}_{i,k} - \circled{$5$}_{k,k,k,i}$. So we just need to show that the left part of this equality is also in  $\tau_2(\AJ)$ to conclude that $\circled{$7$}_{i,k,k,i}$ is realizable. This comes once again from Lemma \ref{morformula} and the fact that the curves bounding $(\alpha_i \sharp \alpha_k) \cup (\beta_k \sharp\alpha_i)$ and $(\alpha_k) \cup (\beta_k \sharp \alpha_i)$ (where $(\alpha_i \sharp \alpha_k)$ and $(\beta_k \sharp\alpha_i)$ are connected sums of the curves involved) are bounding disks in the handlebody. All these computations imply that $P \subset \tau_2(\AJ)$, so that $T_{\geq 3} \in \tau_2(\AJ)$.

Consequently, we get that $T \in \tau_2(\AJ)$ which finishes the proof of Theorem \ref{lastthm}.

\begin{rem}
The computations in this section actually give generators for $\tau_2(\AJ)$, which we can write explicitly. Also, it can be noticed that we used only elements of $\AJ$ of the first and the third kind defined in Section \ref{section51}. This tells us something about the generation of $\AJ$, but only up to $J_3$. Naturally the following question arises: is $\AJ$ generated by elements of the first and the third kind in Section \ref{section51} ?
\label{finalremark}
\end{rem}

Theorem \ref{lastthm} allows us to recover the result shown by Pitsch in \cite{pit}, whose immediate corollary is that any homology 3-sphere is $J_3$-equivalent to $S^3$. We even get a slight improvement on the genus condition. With the definitions of $\mathcal{A}$, $\mathcal{B}$, and $\iota$ given in Section \ref{sec3}, we get the following result :

\begin{coroll}
For any $g \geq 4$, $\operatorname{Im}(\tau_2) = \tau_2(\AJ) + \tau_2(\mathcal{B} \cap J_2)$.
\end{coroll}

\begin{proof}
Any element $T$ in the image of $\tau_2$ can be written as (an expansion of) a linear combination $T_1$ of trees with $0$ or $1$ leaf colored by $A$ and a linear combination $T_2$ of trees with $2,3$ or $4$ leaves colored by $A$ (here, the term ``tree" includes halfs of symmetric trees as well). Then it is clear that $\operatorname{Tr}^{as}(T_1) \in B \wedge B$, whereas $\operatorname{Tr}^{as}(T_2) \in A \wedge H$. The spaces $B \wedge B$ and $A \wedge H$ having trivial intersection in $\Lambda^2 H$, both $T_1$ and $T_2$ lie in the kernel of $\operatorname{Tr}^{as}$. The term $T_2$, by definition, also lies in the kernel of $\operatorname{Tr}^A$. We also know (see Section \ref{sec6}) that $\iota$ acts on $H$ as the map sending $a_i$ to $(-b_i)$ and $b_i$ to $a_i$ for all $i$'s, and that $\mathcal{B}= \iota \mathcal{A} \iota^{-1}$. Now $\iota_*(T_1)$ lie in the kernel of $\operatorname{Tr}^A$. By Theorem \ref{theoremtrA}, we know that there are two mapping classes $\psi_1$ and $\psi_2$ in $\AJ$ such that $T= T_1 + T_2 = \tau_2(\iota\psi_1 \iota^{-1}) + \tau_2(\psi_2)$, which finishes the proof.
\end{proof}

\begin{rem}
In this proof, we used only the fact that $\mathcal{F}_1 \cap \tau_2(J_2) \subset \tau_2(\AJ)$ which is stricly weaker than the equality $\tau_2(\AJ) = \operatorname{Ker}(\operatorname{Tr}^A) \cap \tau_2(J_2)$ from Theorem \ref{lastthm}. In this sense, the computation in this section is more precise than the one from \cite{pit}.
\end{rem}

\section{Computing $\tau_2(\mathcal{G} \cap J_2)$}
\label{sec6}
Like in Section \ref{sec3}, we choose a system of meridians and parallels in the boundary of $V_g$, and we identify $S^3$ to ${V_g} {\cup}_{\iota_g} (-V_g)$. This gives the Heegaard splitting of genus $g$ of the 3-sphere, and we consider the subgroup $\mathcal{B}= \iota \mathcal{A} \iota^{-1}$  of $\mathcal{M}$. We thus have a family of curves $(\alpha_i)_{1 \leq i \leq g}$ with homology classes $(a_i)_{1 \leq i \leq g}$ as in the previous sections, but also a set of curves $(\beta_i)_{1 \leq i \leq g}$ with homology classes $(b_i)_{1 \leq i \leq g}$, defining a Lagrangian $B \subset H$. The map $\iota$ can be defined by its action on $\pi$. We lift the curves $\alpha_i$ and $\beta_i$ to elements of $\pi$ as described in Figure \ref{fig5}, and we set \[\begin{aligned}
{\iota}_*: \pi & \longrightarrow \pi \\
\alpha_{i} & \longmapsto \beta_{i}^{-1} \\
\beta_{i} & \longmapsto \beta_{i} \alpha_{i} \beta_{i}^{-1}.
\end{aligned}\] Indeed by the Dehn-Nielsen theorem, as $\iota_*$ fixes the element $\xi := \prod_{i=1}^g [\beta_{i}^{-1},\alpha_{i}]$ defined by $-\partial \Sigma$ in $\pi$ ($\xi$ is described in Figure \ref{fig5}), the map $\iota$ realizing this action is well-defined.
\begin{figure}[h]
	\centering
	\includegraphics[scale= 1]{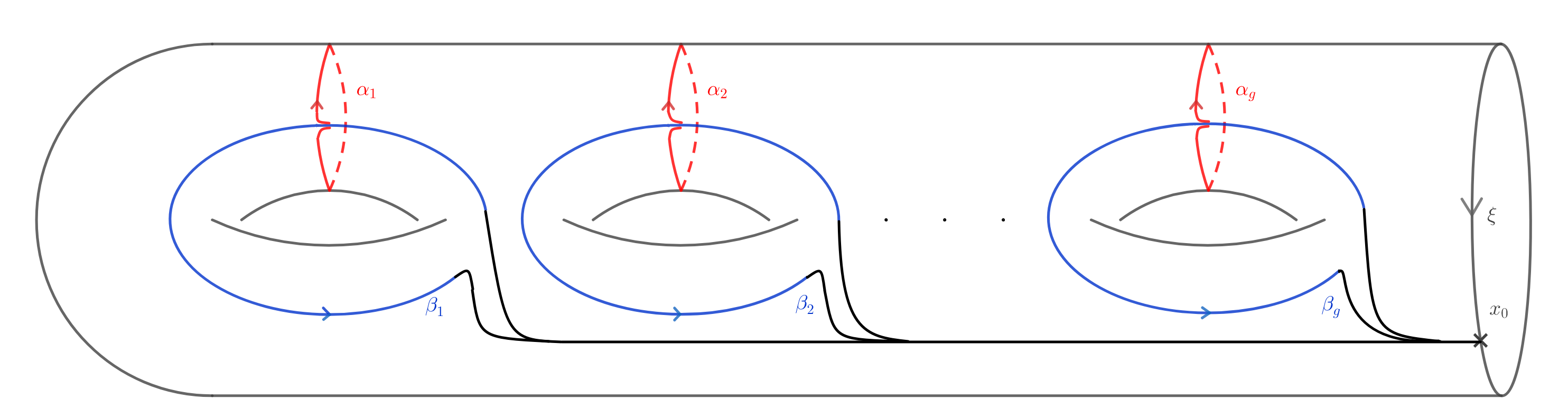}
	\caption{The based curves $(\alpha_i)_{1 \leq i \leq g}$ and $(\beta_i)_{1 \leq i \leq g}$}
	\label{fig5}
\end{figure}

The \textit{Goeritz group} of $S^3$ is the group of isotopy classes of orientation-preserving homeomorphisms of $S^3$ preserving this Heegaard splitting (and fixing the disk). We denote it by $\mathcal{G}:=\mathcal{G}_{g,1}$. Observe that $\mathcal{G}$ coincides with $\mathcal{A} \cap \mathcal{B}$. The Johnson filtration restricts to a separating filtration on $\mathcal{G}$. In this section, we compute $\tau_1(\mathcal{G}\cap J_1)$ and $\tau_2(\mathcal{G}\cap J_2)$ using, respectively, a refinement of the computations made by Morita in \cite{mor} and the computations and results in Section \ref{sec5}. Notice that $\iota$ acts on $\mathcal{G}$ by conjugation, and on $H$ by sending $a_i$ to $-b_i$ and $b_i$ to $a_i$ for all $i$. We also need the following from \cite[Section~3]{suz}.

\begin{lemma}
The image of $\mathcal{G}$ in $\operatorname{Sp}({2g,\mathbb{Z}})$ coincides with \[\left\{\left(\begin{array}{cc}
P & 0 \\
0 & \left(P^{T}\right)^{-1}
\end{array}\right) \bigg \lvert \quad P \in \mathrm{GL}(g, \mathbb{Z})\right\}\]  and, so, is canonically isomorphic to $\operatorname{GL}(g, \mathbb{Z})$.
\label{tau0goeritz}
\end{lemma}

Thus, for all $k$, $\tau_k(\mathcal{G}\cap J_k)$ is a $\mathrm{GL}(g, \mathbb{Z})$-module.

\begin{propo} 
For $g \geq 2$, we have $\tau_1(\mathcal{G}\cap J_1) = A \wedge B \wedge H$.
\label{tau1goeritz}
\end{propo}

\begin{proof}
We identify once again elements of $\Lambda^3 H$ to trees with three leaves. Any element in $\tau_1(\mathcal{G}\cap J_1)$ must vanish when we reduce its leaves in $H/A$ or $H/B$. Hence it can be written as a linear combination of trees whose leaves are never colored solely by $A$ or by $B$. Now, one can check that any tree in $A \wedge B \wedge H$ colored by elements in $ \{ a_i, b_i \mid 1 \leq i \leq g \}$ is in the $\mathbb{Z}$-module generated by the orbit of $T := \ltritree{$a_1$}{$b_1$}{$b_2$}$ under the actions of $\iota$ and $\mathrm{GL}(g, \mathbb{Z})$. Indeed, if such a tree has 2 leaves colored by $A$, the action of $\iota$ allows us to have a tree in the same orbit but with two leaves colored by $B$. Now, such a tree is always in the orbit of $T$ or $T' := \ltritree{$a_1$}{$b_2$}{$b_3$}$ under the action of $\mathrm{GL}(g, \mathbb{Z})$ (just by renumbering). But $T'$ is also in the $\mathbb{Z}$-module generated by the orbit of $T$, as one can write $T' = \ltritree{$a_1$}{$\quad b_1+b_2$}{$b_3$}-T$. Hence, it is sufficient to show that this particular tree is in $\tau_1(\mathcal{G}\cap J_1)$. Actually, if $\psi$ denotes the composition of a right Dehn twist along a simple closed curve corresponding to $[\alpha_2,\beta_2^{-1}][\alpha_1,\beta_1^{-1}]\beta_2 \in \pi$ with the left Dehn twist along a simple closed curve corresponding to $\beta_2 \in \pi$ (as described in Figure \ref{fig6} and in Fig. 3a in \cite{mor}), then $\tau_1(\psi) = T$. The map $\psi$ is an annulus twist in the inner handlebody, and the composition of two Dehn twists along curves bounding disks in the outer handlebody. Hence, we have $\psi \in \mathcal{A} \cap \mathcal{B} = \mathcal{G}$.
\end{proof}
\begin{figure}[h]
	\centering
	\includegraphics[scale= 1]{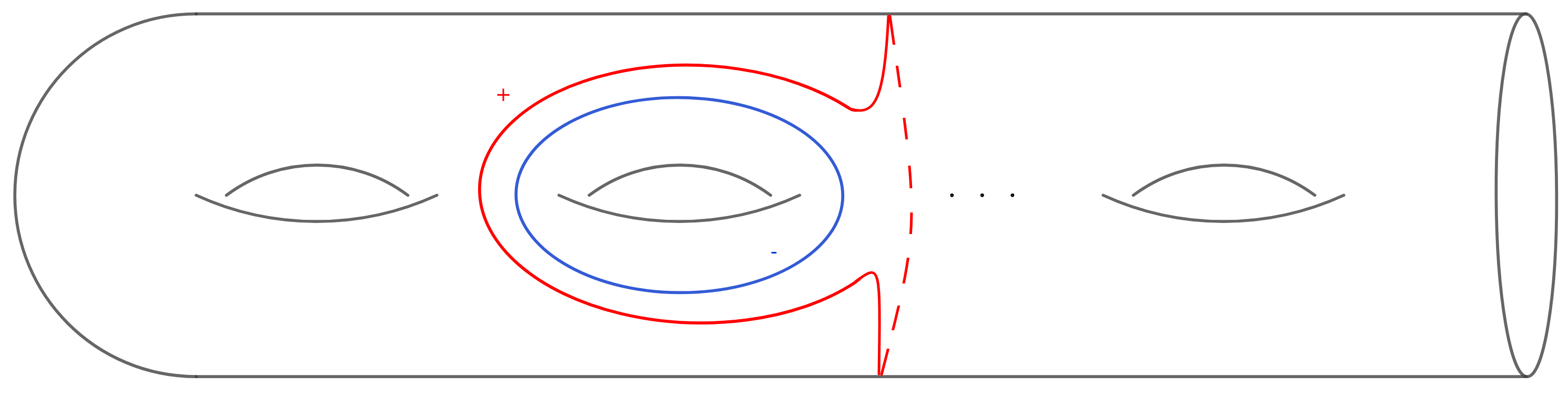}
	\caption{The two curves defining $\psi$}
	\label{fig6}
\end{figure}

Even though $\operatorname{Tr}^A$ and $\operatorname{Tr}^B$ are not defined over the same subspaces of $D_2(H)$, their kernels are both included in $D_2(H)$. Hence the following makes sense
\begin{propo}
For $g \geq 4$, we have $\tau_2(\mathcal{G} \cap J_2) =  \operatorname{Ker}( \operatorname{Tr}^{as}) \cap \operatorname{Ker}( \operatorname{Tr}^A) \cap \operatorname{Ker}( \operatorname{Tr}^B)$.
\label{tau2goeritz}
\end{propo}

\begin{proof}
The inclusion from the left to the right is a direct consequence of Theorems \ref{theoremtraces} and \ref{lastthm}. For the other inclusion, let us take an element $T \in \operatorname{Ker}( \operatorname{Tr}^{as}) \cap \operatorname{Ker}( \operatorname{Tr}^A) \cap \operatorname{Ker}( \operatorname{Tr}^B)$. We say, for $0 \leq k,l \leq 4$, that a tree has \textit{type} $(k,l)$ if it has $k$ leaves colored by $A$ and $l$ leaves colored by $B$. The tree $T$ is then a linear combination of trees of type $(1,3)$, $(2,2)$ and $(3,1)$. This is due to the fact that $T$ must be in the kernel of the projections $D_2(H) \rightarrow D_2(H/A) $ and $D_2(H) \rightarrow D_2(H/B)$. Hence we decompose $T$ into 3 elements: $T = T_1 + T_2 + T_3$, such that $T_i$ is a linear combination of elements of type $(i,4-i)$, for $i=1,2,3$. The images by $ \operatorname{Tr}^{as}$ of these 3 elements take place in separate summands in $\Lambda^2(H/2H)$. Also, by definition, $\operatorname{Tr}^{A}$ (resp.  $\operatorname{Tr}^{B}$) vanishes on the element of types $(2,2)$ and $(3,1)$ (resp. $(2,2)$ and $(1,3)$). We thus have, for $i=1,2,3$, $T_i \in \operatorname{Ker}( \operatorname{Tr}^{as}) \cap \operatorname{Ker}( \operatorname{Tr}^A) \cap \operatorname{Ker}( \operatorname{Tr}^B)$. We thus treat these 3 elements separately.

The element $T_1$ belongs to the space $N$ defined in Proposition \ref{defN}. By Remark \ref{remcoloration}, any element in $N$ can be realized as a linear combination of commutators of trees with three leaves with always at least one leaf in $A$ and one leaf in $B$. By Proposition \ref{tau1goeritz}, this implies that $N \subset \tau_2(\mathcal{G}\cap J_2)$. Indeed, we have $[\mathcal{G} \cap J_1,\mathcal{G} \cap J_1] \subset \mathcal{G} \cap J_2$.

The element $T_3$ is such that $\iota_* (T_3)$ is of type $(1,3)$, hence is in $N$. We deduce that $\iota_* (T_3)$, and consequently $T_3$, also belong to $\tau_2(\mathcal{G}\cap J_2)$.

The element $T_2$ is exactly of the same type as its homonym in Section \ref{sec5}. We want to modify slightly the argument in order to show that it is also in $\tau_2(\mathcal{G}\cap J_2)$. The elements $[\alpha_1,\beta_1^{-1}]$ and $[\alpha_2,\beta_2^{-1}][\alpha_1,\beta_1^{-1}]$ in $\pi$ (where the curves have been lifted to elements of $\pi$ as in Figure \ref{fig5}) define two simple closed curves bounding disks both in the inner and the outer handlebody. Then the Dehn twists along these two curves are maps in the Goeritz group, but also in $J_2$. This gives, respectively, that ${\circled{6}}_{1,1}$ and ${\circled{6}}_{1,1} -{\circled{5}}_{1,1,2,2} + {\circled{6}}_{2,2}$ are in $\tau_2(\mathcal{G}\cap J_2)$. Using the action of $\mathrm{GL}(g, \mathbb{Z})$ (by sending $1$ on $i$ and $2$ on $j$), we deduce that for all $ 1 \leq i \neq j \leq g$,  we have ${\circled{6}}_{i,i}$, ${\circled{5}}_{i,i,j,j} \in \tau_2(\mathcal{G}\cap J_2)$. For the trees of type $\circled{5}_{i,j,k,l}$ with no contraction discussed page 24, we can simply write \[\ltree{$a_i$}{$b_j$}{$b_k$}{$a_l$} = \Bigg[\ltritree{$a_i$}{$b_j$}{$a_l$} ,\ltritree{$a_l$}{$b_l$}{$b_k$} \Bigg]\] if all the indices are different or\[\ltree{$a_i$}{$b_j$}{$b_k$}{$a_l$} = \Bigg[\ltritree{$a_i$}{$b_j$}{$a_m$} ,\ltritree{$a_l$}{$b_m$}{$b_k$} \Bigg]\] with $m \notin \{ i,j,k,l \}$ otherwise (which is possible with $g \geq 4$). We conclude, as in Section \ref{sec5} that $T_2$ is a sum of an element in $\tau_2(\mathcal{G}\cap J_2)$ and an element $T_2' = U +V$ with $U \in K$ and $V \in S$, where these spaces are respectively defined in the short exact sequence \eqref{defK} and in Lemma \ref{lemmeS}. The computations showing that $V \in \tau_2(\AJ)$ only involves commutators of trees colored both by $A$ and $B$, and the element ${\circled{5}}_{j,j,1,1}$. By Proposition \ref{tau1goeritz}, $V \in \tau_2(\mathcal{G}\cap J_2)$.
Finally, using once again the same argument, we only need to show that $\circled{6}_{l,j} - \circled{5}_{j,j,j,l} \in \tau_2(\mathcal{G}\cap J_2)$ for $i \neq j$ to show that $U \in \tau_2(\mathcal{G}\cap J_2)$. We notice that the action of $\mathrm{GL}(g, \mathbb{Z})$ corresponding to $b_1 \mapsto b_1 + b_2$ and $a_2 \mapsto a_2-a_1$ (and fixing the other elements in the basis) sends $\circled{6}_{1,2}$ to $\circled{6}_{1,2} - \circled{5}_{1,1,2,1}$. Then the action of $\iota$ on this element gives $\circled{6}_{2,1} - \circled{5}_{1,1,1,2}$. This proves that $\circled{6}_{2,1} - \circled{5}_{1,1,1,2} \in \tau_2(\mathcal{G}\cap J_2)$, and by action of $\mathrm{GL}(g, \mathbb{Z})$, that $\circled{6}_{l,j} - \circled{5}_{j,j,j,l} \in \tau_2(\mathcal{G}\cap J_2)$ for any $i \neq j$.
\end{proof}

\begin{rem}
The rationalization of $\tau_2(\G \cap J_2)$ is a finite-dimensional $\mathrm{GL}(g, \mathbb{Q})$-module. In Appendix A, we give its decomposition into irreducible modules. This results in a rational version of Proposition \ref{tau2goeritz}.
\end{rem}

Clearly one has that $\tau_k(\G \cap J_k) \subset \tau_k(\A \cap J_k) \cap \tau_k(\mathcal{B} \cap J_k)$. It is not clear if the converse is true in general. As a direct consequence of Proposition \ref{tau1goeritz}, Proposition \ref{tau2goeritz} and Theorem \ref{lastthm}, we get the following:

\begin{coroll}For $g \geq 4$, we have $\tau_1(\G \cap J_1) = \tau_1(\A \cap J_1) \cap \tau_1(\mathcal{B} \cap J_1)$ and $\tau_2(\G \cap J_1) = \tau_2(\A \cap J_2) \cap \tau_2(\mathcal{B} \cap J_2)$.
\label{cor65}
\end{coroll}

In \cite{pitco}, Pitsch already pointed out that the image of $\G$ in $\operatorname{Sp}({2g,\mathbb{Z}})$ coincides with the intersection of the images of $\A$ and $\mathcal{B}$ (see Lemma \ref{tau0goeritz}). Using this fact and the Reidemeister-Singer Theorem (see Theorem \ref{rsing}), he showed (\cite[Theorem 1]{pitco}):

\begin{propo}
There is a well-defined isomorphism \[\lim_{g \rightarrow \infty} \big ( (\mathcal{A}_{g, 1} \cap \mathcal{I}_{g, 1}) \backslash \mathcal{I}_{g, 1} /  (\mathcal{B}_{g, 1} \cap \mathcal{I}_{g, 1}) \big )_{\G_{g, 1}} \simeq \mathcal{S}^{3}.\] where $\G$ acts by conjugation.
\label{propitsch}
\end{propo}

This gives an intrinsic description of the equivalence relation given by Reidemeister-Singer Theorem on the Torelli group. The same can be done, using Corollary \ref{cor65}, for the second and third term of the Johnson filtration.

\begin{propo}
Denote $\K_{g,1}:= J_2(\Sigma_{g,1})$ and $\mathcal{L}_{g,1}:= J_3(\Sigma_{g,1})$. There are well-defined isomorphisms \[\lim_{g \rightarrow \infty} \big ( (\mathcal{A}_{g, 1} \cap \K_{g,1}) \backslash \K_{g,1} /  (\mathcal{B}_{g, 1} \cap \K_{g,1}) \big )_{\G_{g, 1}} \simeq \mathcal{S}^{3}.\]

and

\[\lim_{g \rightarrow \infty} \big ( (\mathcal{A}_{g, 1} \cap \mathcal{L}_{g,1}) \backslash \mathcal{L}_{g,1} /  (\mathcal{B}_{g, 1} \cap \mathcal{L}_{g,1}) \big )_{\G_{g, 1}} \simeq \mathcal{S}^{3}.\]
\label{prop66}
\end{propo}

\begin{proof}
The proof is by induction. We already know that the maps are well-defined and surjective (see Section \ref{sec3}. We know by Proposition \ref{propitsch} that two gluing maps $\phi \in \mathcal{K}_{g,1}$ and $\psi \in \mathcal{K}_{g,1}$ yield the same homology 3-sphere if and only if, after an eventual stabilization, there exists maps $\xi_a \in \A \cap \I$, $\xi_b \in \B \cap \I$ and $\mu \in \G$ such that $\phi=\mu \xi_{a} \psi \xi_{b} \mu^{-1}$. Applying $\tau_1$ to this equality we get that $\tau_1(\xi_a) = - \tau_1(\xi_b) \in \tau_1(\A \cap \I) \cap \tau_1(\B  \cap \I) = \tau_1( \G \cap \I)$. Then there exists $ \mu' \in \G \cap \I$ such that $ \mu^{-1}\phi \mu =\mu' \circ\left(\mu'^{-1} \xi_{a}\right) \psi \left(\xi_{b} \mu'\right) \circ \mu'^{-1}$, and $\mu'^{-1} \xi_{a} \in \A \cap \mathcal{K}_{g,1}$, $\xi_{b} \mu' \in \B \cap \mathcal{K}_{g,1}$. Then a conjugate of $\phi$ by an element of the Goeritz group is in the same double coset class as $\psi$. This concludes, as one can get the proof for $\mathcal{L}_{g,1}$ by applying the same method to some elements $\phi$ and $\psi$ in $\mathcal{L}_{g,1}$.
\end{proof}

Using the methods described by Pitsh in \cite{pitco}, Proposition \ref{prop66} could help to build invariants of homology 3-spheres by using algebraic properties of $J_2$ and $J_3$. Infortunately we do not know about generators of $\A \cap J_2$ and $\A \cap J_3$.
\newpage
\begin{appendices}
\titlelabel{\normalsize APPENDIX \thetitle.  }
\section{Decomposition of $\tau_2(\G \cap J_2) \otimes \Q$}

As a consequence of Lemma \ref{tau0goeritz}, the conjugation action of the Goeritz group $\G$  on itself induces a $\mathrm{GL}(g, \mathbb{Z})$-module structure on $\tau_2(\G \cap J_2)$, the image of the Goeritz group by the second Johnson homomorphism. This action is the restriction of the canonical action of $\mathrm{GL}(g, \mathbb{Z}) \subset \operatorname{Sp}(2g, \mathbb{Z}) \simeq \operatorname{Sp}(H)$ on $D_2(H)$ to $\tau_2(\G \cap J_2)$. Let $D_2(H^\Q)$ be the rationalization of the abelian group $D_2(H)$, with $H^\Q := H \otimes \Q$. It is clear that $H^\Q$ is a $\mathrm{GL}(g, \mathbb{Q})$-module, hence $D_2(H^\Q)$ is also a $\mathrm{GL}(g, \mathbb{Q})$-module. Then, by standard arguments (see \cite{an} for instance), the $\mathrm{GL}(g, \mathbb{Z})$-module structure on $\tau_2(\G \cap J_2)$ extends to a $\mathrm{GL}(g, \mathbb{Q})$-module structure on $\tau_2(\G \cap J_2)\otimes \Q$.

In this appendix, we fix a genus $g \geq 4$ and we give the decomposition of this module into irreducible $\mathrm{GL}(g, \mathbb{Q})$-modules. We do not use any results from Section \ref{sec5}. Recall from Section 6 that we have a basis $(a_i,b_i)_{1 \leq i \leq g}$ for $H$, inducing a basis for $H^\Q$. This also yields a decomposition $H = A^\Q \oplus B^\Q$, with $A^\Q$ and $B^\Q$ stable under the action of $\mathrm{GL}(g, \mathbb{Q})$. Specifically, $\operatorname{GL}(A^\Q)$ acts on $A^\Q$ and $(A^\Q)^*$ in the natural way, $B^\Q$ is identified to $(A^\Q)^*$ via $\omega$ and $\operatorname{GL}(A^\Q)$ is identified to $\operatorname{GL}(g,\Q)$ through the basis $(a_1, \dots, a_g)$ of $A^\Q$.


Let $D_{i,j}$ be the subspace of $D_2(H^\Q)$ generated by expansions of trees with $i$ leaves in $A^\Q$ and $j$ leaves in $B^\Q$, for $0 \leq i \leq 4$ and $i+j = 4$. We compute the dimensions of these submodules of $D_2(H^\Q)$.

\begin{lemma}For any $g \geq 3$, we have: 
\begin{align}
\operatorname{dim}(D_2(H^\Q)) &= \frac{g^2(2g-1)(2g+1)}{3}\label{A1}\\
\operatorname{dim}(D_{i,j}) &= \operatorname{dim}(D_{j,i})\label{A2}\\
\operatorname{dim}(D_{0,4}) &= \frac{g^2(g-1)(g+1)}{12}\label{A3}\\
\operatorname{dim}(D_{1,3}) &= \frac{g^2(g-1)(g+1)}{3}\label{A4}\\
\operatorname{dim}(D_{2,2}) &= \frac{g^2(g^2+1)}{2}\label{A5}
\end{align}

\begin{proof}
Equation \eqref{A1} is a consequence of the isomorphism $\frac{(\Lambda^2 H^\Q \otimes \Lambda^2 H^\Q)^{\mathfrak{S}_2}}{\Lambda^4 H^\Q} \simeq  D_2(H^\Q)$ (see diagram \eqref{diagmas}). Equation \eqref{A2} is obtained by interchanging the $a_i's$ and $b_i's$. We also notice that $D_{0,4} \simeq D_2(A^\Q)$, and we obtain equation \eqref{A3}. The space $D_{1,3}$ is isomorphic to $A^\Q \otimes \mathcal{L}_3(B^\Q)$, and the dimension of $\mathcal{L}_3(V)$ is equal to $\frac{n^3-n}{3}$ for a vector space $V$ of dimension $n$: this proves equation \eqref{A4}. Equation \eqref{A5} follows from the previous using that $D_2(H^\Q)= \bigoplus_{0 \leq i \leq 4} D_{i,4-i}$. One can also get equation \eqref{A5} by showing that there is an isomorphism $D_{2,2}\simeq S^2(A^\Q \otimes B^\Q)$.
\end{proof}
\end{lemma}

We now decompose $\tau_2(\G \cap J_2) \otimes \Q$ in 3 submodules and compute their respective dimensions. Denote $\operatorname{Tr}^{A,\Q}$ for $\operatorname{Tr}^A \otimes \Q$ and $\operatorname{Tr}^{B,\Q}$ for $\operatorname{Tr}^B \otimes \Q$. The kernels of these two maps are both regarded as $\operatorname{GL}(g,\Q)$-submodules of $D_2(H^\Q)$.

\begin{coroll}
The space $\tau_2(\G \cap J_2)\otimes \Q$ is a subset of $$ \operatorname{Ker}(\operatorname{Tr}^{A,\Q}) \cap \operatorname{Ker}(\operatorname{Tr}^{B,\Q})= \big( D_{1,3} \cap \operatorname{Ker}(\operatorname{Tr}^{A,\Q}) \big )\oplus D_{2,2} \oplus \big( D_{3,1}\cap \operatorname{Ker}(\operatorname{Tr}^{B,\Q})\big),$$ and the summands are $\operatorname{GL}(g,\Q)$-submodules with respective dimensions $\frac{g(g+1)(2g^2-2g-3)}{6}$, $\frac{g^2(g^2+1)}{2}$ and $\frac{g(g+1)(2g^2-2g-3)}{6}$.
\end{coroll}

\begin{proof}
The inclusion is a consequence of Theorem \ref{theoremtrA}, given that $\G= \A \cap \B$. The decomposition is an immediate consequence of the fact that $D_{3,1} \oplus D_{2,2} \subset \operatorname{Ker}(\operatorname{Tr}^{A,\Q})$, and $D_{1,3} \oplus D_{2,2} \subset \operatorname{Ker}(\operatorname{Tr}^{B,\Q})$. The maps $\operatorname{Tr}^{A,\Q}$ and $\operatorname{Tr}^{B,\Q}$ respect the action of $\operatorname{GL}(g,\Q)$ by Remark \ref{reminvariance}, hence the 3 summands are $\operatorname{GL}(g,\Q)$-submodules. The computation of the dimensions is a consequence of the rank theorem and the previous lemma, as $\operatorname{Tr}^{A,\Q}$ is surjective onto $S^2(H/A \otimes \Q)$. 
\end{proof}

Next, we use the representation theory of $\mathrm{SL}(g, \mathbb{C})$, and exhibit the irreducible modules in $\tau_2(\G \cap J_2) \otimes \Q$ by finding heighest weight vectors. Our notation convention for a Young diagram with $n$ rows of type $(\lambda_1 \geq \lambda_2 \geq  ...  \geq \lambda_n \geq 0$) is $[\lambda_1\lambda_2...\lambda_n]$. To such a diagram is associated an irreducible representation of $\mathrm{SL}(g, \mathbb{\Q})$ whenever $n \leq g-1$, as described in \cite{fulhar}. For short, to a Young diagram $\lambda := [\lambda_1\lambda_2...\lambda_{g-1}]$ is associated the subrepresentation of the tensor product $\bigotimes_{i=1}^{g-1} S^{(\lambda_i- \lambda_{i+1})} (\Lambda^{i}V)$ spanned by $v_{\lambda} := (e_1)^{\lambda_1 - \lambda_2}\otimes (e_1\wedge e_2)^{(\lambda_2 - \lambda_3)}  \otimes \ldots  \otimes (e_1 \wedge \ldots \wedge e_{g-1})^{(\lambda_{g-1} - \lambda_{g})}$, where $V:=\Q^g$ has a basis $e_1, e_2 \dots e_g$, and $\lambda_g= 0$. This defines both a representation of $\operatorname{GL}(g,\Q)$ and $\operatorname{SL}(g,\Q)$.

\begin{theorem}
For any $g \geq 4$, we have an isomorphism of $\operatorname{SL}(g,\Q)$-modules $$\tau_2(\G \cap J_2)\otimes \Q = 2[0] + 2[21^{g-2}] + [42^{g-2}] + [2^21^{g-4}] + [32^{g-3}1]+[1^{g-2}] + [321^{g-3}] +[1^2].$$
\label{thsl}
\end{theorem}

\begin{proof}[Sketch of proof]
We simply need to exhibit highest weight vectors in $\tau_2(\G \cap J_2)\otimes \Q$ for the action of $\mathrm{SL}(g, \mathbb{\Q})$ on $D_2(H^\Q)$, such that the sum of the dimensions of the modules they generate is the dimension of $\operatorname{Ker}(\operatorname{Tr}^{A,\Q}) \cap \operatorname{Ker}(\operatorname{Tr}^{B,\Q})$. We can check this using Lemma \ref{A1}, and it is standard representation theory to verify that a given vector is a highest weight vector. Hence we get that $D_{2,2}$ decomposes into \[
 \begin{array}{|c|c|c|c|c|c|c|}
 \hline \bigoplus & [0] &[0]&[21^{g-2}]& [21^{g-2}] & [42^{g-2}]& [2^21^{g-4}] \\
\hline  \operatorname{dim} & 1 &1 & g^2-1 & g^2-1 & \frac{g^2(g-1)(g+3)}{4} & \frac{g^2(g+1)(g-3)}{4} \\
\hline \operatorname{HWV}& \tau_2(T_\xi)& \sum\limits_{i,j = 1}^{g}\vltree{$a_i$}{$a_j$}{$b_i$}{$b_j$} & \sum\limits_{i = 1}^{g}\vltree{$a_1$}{$b_g$}{$b_i$}{$a_i$}&\sum\limits_{i = 1}^{g}\vltree{$a_1$}{$a_i$}{$b_i$}{$b_g$} & \vltree{$a_1$}{$b_g$}{$a_1$}{$b_g$}& \vltree{$a_1$}{$a_2$}{$b_g$}{$b_{g-1}$}\\\hline
\end{array}\]

\noindent where $T_\xi$ is the Dehn twist around the boundary component of $\Sigma_{g,1}$, and $\big(D_{1,3} \cap \operatorname{Ker}(\operatorname{Tr}^{A,\Q})\big) \oplus \big(D_{3,1} \cap \operatorname{Ker}(\operatorname{Tr}^{B,\Q})\big)$  decomposes into 
\[\begin{array}{|c|c|c||c|c|}
 \hline \bigoplus & [32^{g-3}1] & [1^{g-2}] & [321^{g-3}]& [1^2] \\
\hline  \operatorname{dim} & \frac{g^2(g-2)(g+2)}{3} & \frac{g(g-1)}{2} & \frac{g^2(g-2)(g+2)}{3} & \frac{g(g-1)}{2} \\
\hline \operatorname{HWV}& \vltree{$a_1$}{$b_g$}{$b_{g-1}$}{$b_{g}$} & \sum\limits_{i = 1}^{g} \vltree{$a_{i}$}{$b_{i}$}{$b_g$}{$b_{g-1}$}& \vltree{$a_2$}{$a_1$}{$b_g$}{$a_1$} & \sum\limits_{i = 1}^{g} \vltree{$a_1$}{$a_2$}{$b_i$}{$a_i$}\\\hline
\end{array}.\] One can also get these decompositions by giving tensorial description of the modules (such as $D_{2,2}\simeq S^2(A^\Q \otimes B^\Q)$) and by using Pieri's formula. It remains to show that the above highest weight vectors are indeed in $\tau_2(\G \cap J_2)\otimes \Q$. The author checked this for $g \geq 4$ and did it in the same spirit as in the proof of Proposition \ref{tau2goeritz}.
\end{proof}

This already gives a rational version of Proposition \ref{tau2goeritz}. 
\begin{coroll}
For any $g \geq 4$, $\tau_2(\G \cap J_2)\otimes \Q = \operatorname{Ker}(\operatorname{Tr}^{A,\Q}) \cap \operatorname{Ker}(\operatorname{Tr}^{B,\Q})$.
\end{coroll}

Finally, we turn to the decomposition of $\tau_2(\G \cap J_2) \otimes \Q$ into irreducible $\operatorname{GL}(g,\Q)$-modules. For any integer $k\geq 0$, we now denote $\operatorname{Det}^k$ the $k$th power of the determinant representation, and $\operatorname{Det}^{-k}$, its dual. Any irreducible rational representation of $\mathrm{GL}(g, \mathbb{C})$ is obtained as the tensor product of an irreducible representation of $\mathrm{SL}(g, \mathbb{C})$ of type $\lambda$ (for a young diagram $\lambda$) with a power of the determinant representation. By looking at the action of the center of $\operatorname{GL}(g,\Q)$ on the highest weight vectors given in the proof of Theorem $\ref{thsl}$, we get the following:

\begin{theorem}
For any $g \geq 4$, we have an isomorphism of $\operatorname{GL}(g,\Q)$-modules
\begin{align*}
\tau_2(\G \cap J_2)\otimes \Q &= 2[0] + 2[21^{g-2}]\otimes \operatorname{Det}^{-1}  + [42^{g-2}]\otimes \operatorname{Det}^{-2} + [2^21^{g-4}]\otimes \operatorname{Det}^{-1} \\&+ [32^{g-3}1]\otimes \operatorname{Det}^{-2}+[1^{g-2}]\otimes\operatorname{Det}^{-1} + [321^{g-3}]\otimes\operatorname{Det}^{-1} +[1^2].
\end{align*}
\label{gldec}
\end{theorem}

\begin{proof}
We know that each irreducible summand $W$ of the $\operatorname{SL}(g,\Q)$-module decomposition of $\tau_2(\G \cap J_2) \otimes \Q$ is isomorphic as a $\operatorname{GL}(g,\Q)$-module to $W \otimes \operatorname{Det}^k$, for some $k \in \mathbb{Z}$. We also know that the isomorphism between the ``model'' representation given by the Young diagram $\lambda$ and $W$ can be made explicit by sending $v_\lambda$ to the highest weight vector of our representation. The integer $k$ must be chosen in such a way that this isomorphism lifts to an isomorphism of $\operatorname{GL}(g,\Q)$-modules.

We only do the computation for one summand, say $\lambda =[32^{g-3}1]$. The map sending $v_\lambda$ to $T_\lambda := \vltree{$a_1$}{$b_g$}{$b_{g-1}$}{$b_{g}$}$ is an isomorphism of $\operatorname{SL}(g,\Q)$-modules, but one can check that for any $d \in \Q$, $(d Id) \cdot v_\lambda = d^{2g-2}v_\lambda$, while $(d Id) \cdot T_\lambda = \frac{1}{d^2} T_\lambda$. By choosing $k = -2$, we get that the map from $[32^{g-3}1] \otimes \operatorname{Det}^{-2}$ to the $\operatorname{GL}(g,\Q)$-module spanned by $T_\lambda$, sending $v_\lambda \otimes 1$ to $T_\lambda$ is a $\operatorname{GL}(g,\Q)$-equivariant isormorphism. More generally, one can check that for a Young diagram $\lambda:= [\lambda_1\lambda_2...\lambda_{g-1}]$ appearing in the irreducible decomposition of $D_{i,j}$, we get $k = \frac{1}{g}(i-j-\sum\limits_{i = 1}^{g-1} \lambda_i)$.
\end{proof}

\begin{rem}
In the decomposition of Theorem \ref{gldec}, the action of $\iota$ induces the following symmetries:
\begin{enumerate}
\item the irreducible summands in $D_{2,2}$ are isomorphic to their own duals,
\item the irreducible summands in $D_{1,3}$ and $D_{3,1}$ are exchanged when dualizing, indeed we have: $([321^{g-3}] \otimes \operatorname{Det}^{-1})^* \simeq  [32^{g-3}1]\otimes \operatorname{Det}^{-2}$, and $[1^2]^* \simeq [1^{g-2}] \otimes \operatorname{Det}^{-1}$.\end{enumerate}
This is an instance of a general fact: for any $k \geq 1$, $\tau_k(\G \cap J_k)\otimes \Q$ is isomorphic to its dual as a $\operatorname{GL}(g,\Q)$-module. Indeed, the map $\iota$ preserves $\tau_k(\G \cap J_k)\otimes \Q$, and one can see by direct computation that $\forall P \in \operatorname{GL}(g,\Q),  \forall X \in D_k(H^\Q), \iota(P\cdot X) = (P^T)^{-1}\cdot \iota(X)$. Hence the basis of $H^\Q$ induces a $\Q$-module isomorphism between $D_2(H^\Q)$ and its dual, and the composition of this isomorphism with $\iota$ is a $\operatorname{GL}(g,\Q)$-module isomorphism between $D_2(H^\Q)$ and $D_2(H^\Q)^*$. We conclude that if $W$ is an irreducible module in $\tau_k(\G \cap J_k)\otimes \Q$, then $\iota W$ is also an irreducible module in $\tau_k(\G \cap J_k)\otimes \Q$ which is isomorphic as a $\operatorname{GL}(g,\Q)$-module to the dual representation $W^*$.
\label{remA4}
\end{rem}

\label{appA}
\end{appendices}
\bibliographystyle{plain}
\bibliography{bibsjhgv2}
\address
\end{document}